\documentclass[10pt, a4paper]{amsart}

\usepackage{amsmath}
\usepackage{amssymb}
\usepackage{amsthm}
\usepackage{graphicx}
\usepackage{url}
\usepackage{color}
\usepackage{enumitem}

\newtheorem{theorem}{Theorem}[section]
\newtheorem*{theorem*}{Theorem}

\newtheorem{lemma}[theorem]{Lemma}
\newtheorem{proposition}[theorem]{Proposition}

\newtheorem{corollary}[theorem]{Corollary}

\theoremstyle{definition}

\newtheorem{definition}[theorem]{Definition}
\newtheorem{remark}[theorem]{Remark}

\DeclareMathOperator{\red}{red}

\DeclareMathOperator{\ext}{ext}
\DeclareMathOperator{\conv}{conv}

\DeclareMathOperator{\Emp}{\mathfrak{m}}

\DeclareMathOperator{\supp}{supp}
\DeclareMathOperator{\Diff}{Diff}
\DeclareMathOperator{\Per}{Per}%{\mathfrak{Per}}

\DeclareMathOperator{\per}{mp}

\DeclareMathOperator{\M}{\mathcal{M}}
\DeclareMathOperator{\MT}{\mathcal{M}_T}
\DeclareMathOperator{\MS}{\mathcal{M}_S}
\DeclareMathOperator{\CM}{C_\mathcal{M}}

\DeclareMathOperator{\MTe}{\M_T^e}
\DeclareMathOperator{\MTp}{\M_T^{\text{co}}}

\DeclareMathOperator{\MTpos}{\M_T^+}
\DeclareMathOperator{\MTmix}{\M_T^{\text{mix}}}

\DeclareMathOperator{\lang}{\mathcal{L}}

\newcommand{\alf}{\mathcal{A}}
\newcommand{\words}{\alf^*_r}

\DeclareMathOperator{\cch}{\overline{\conv}}
\DeclareMathOperator{\lab}{\Theta}
\DeclareMathOperator{\ini}{i}
\DeclareMathOperator{\term}{t}

\newcommand{\D}{\vec{d}}

\newcommand{\eps}{\varepsilon}

\newcommand{\R}{\mathbb{R}}

\newcommand{\N}{\mathbb{N}}

\newcommand{\db}{B}
\newcommand{\emptyw}{\bot}
\newcommand{\wss}{W^{\textrm{ss}}}
\newcommand{\wsse}{W^{\textrm{ss}}_\eps}
\newcommand{\wsu}{W^{\textrm{su}}}
\newcommand{\wsue}{W^{\textrm{su}}_\eps}

\author{Katrin Gelfert \and Dominik Kwietniak}

\address{
Institute of Mathematics, Federal University of Rio de Janeiro, Cidade
Universitaria - Ilha do Fund\~ao, Rio de Janeiro 21945-909, Brazil}
\email{gelfert@im.ufrj.br}
\address{
Faculty of Mathematics and Computer Science, Jagiellonian University in Krakow, ul. \L o\-jasiewicza 6, 30-348 Krak\'ow, Poland}
\email{dominik.kwietniak@uj.edu.pl}
\urladdr{www.im.uj.edu.pl/DominikKwietniak/}

\thanks{This work was initiated during the visit of DK at the Federal University of Rio de Janeiro and it was finished during the visit of KG to the Jagiellonian University in Krak\'{o}w. The hospitality of these institutions is gratefully acknowledged.
The research of Dominik Kwietniak was supported by the  National Science Centre (NCN) grant Maestro 2013/08/A/ST1/00275. KG has been supported in part by EU Marie-Curie IRSES Brazilian-European partnership in Dynamical Systems (FP7-PEOPLE-2012-IRSES 318999 BREUDS)}

\title[On density of ergodic measures and generic points]{On density of ergodic measures\\ and generic points} % --- an extension of Sigmund's Theorem}
\date{\today}

\begin{document}

\begin{abstract}
We provide conditions which guarantee that ergodic measures are dense in the simplex of invariant probability measures of a dynamical system given by a continuous map acting on a Polish space. Using them we study generic properties of invariant measures and prove that every invariant measure has a generic point. In the compact case, density of ergodic measures means that the simplex of invariant measures is either a singleton of a measure concentrated on a single periodic orbit or the Poulsen simplex. Our properties focus on the set of periodic points and we introduce two concepts:  close\-ability with respect to a set of periodic points and link\-ability of a set of periodic points. Examples are provided to show that these are independent properties. They hold, for example, for systems having the periodic specification property. But they hold also for a much wider class of systems which contains, for example,  irreducible Markov chains over a countable alphabet, all $\beta$-shifts, all $S$-gap shifts, ${C}^1$-generic diffeomorphisms of a compact manifold $M$, and certain geodesic flows of a complete connected negatively curved manifold.
\end{abstract}

\maketitle

\tableofcontents

\section{Introduction}

We study simplices of invariant measures of dynamical systems. By a \emph{dynamical system} $(X,T)$ we mean %%K here
a complete separable metric space $X$ together with a continuous map $T\colon X\to X$. Our goal is to determine when ergodic measures are dense in the space of $T$-invariant Borel probability measures $\MT(X)$.

Recall that if $X$ is compact, then the set $\MT(X)$ equipped with the weak$\ast$  topology is a Choquet simplex and the extreme points of $\MT(X)$ are the ergodic measures. % (see \cite[\S 6.2.]{Walters}).
Choquet simplices generalize the classical %(geometric)
$k$-dimensional simplices (convex hulls of $k+1$ affinely  independent points in $\R^k$). Any point in a Choquet simplex is represented by a unique probability measure on the extreme points. Therefore, if ergodic measures are dense then either $\MT(X)$ is a singleton or a non-trivial Choquet simplex in which extreme points are dense.

The first example of the latter kind was constructed by Poulsen \cite{Poulsen}. An infinite-dimensional Choquet simplex whose extreme points are dense seems to be very far from the intuition built on the finite-dimensional case.
But this is only one of many  remarkable properties of this highly homogeneous object (see Section \ref{sec:Poulsen}). Lindenstrauss, Olsen, and Sternfeld \cite{LOS} proved that a metrizable non-trivial simplex with dense extreme points is unique up to an affine homeomorphism, and therefore, one may talk about  \emph{the} Poulsen simplex. Furthermore, every Choquet simplex appears as a face of the Poulsen simplex.
Downarowicz proved that %\cite{DownarChoquet},
each abstract Choquet simplex is affinely homeomorphic to the set of invariant measures of some minimal compact dynamical system (see \cite{DownarChoquet} for details, references and further historical remarks). % on related results.
%This is , which completes research begun in 50's (see references therein).
%Note that any abstract Choquet simplex is affinely homeomorphic to the set of invariant measures of some minimal dynamical system. This is a result of Downarowicz \cite{DownarChoquet}, which completes research begun in 50's (see references therein).
So the Poulsen simplex may be regarded as the richest possible structure for a set of
invariant measures of a compact dynamical system.

It seems natural to ask (see \cite[p. 654]{P}) when the ergodic measures are dense in the space of all invariant probability measures and, if this is a case,  what  are the properties of a typical invariant measure. % in such a Poulsen simplex.
So far, this question was considered mostly in the compact case  and the specification property (or one of its  generalizations) was the main tool for such investigations.
%, therefore known results apply to systems with some ``uniformly hyperbolic'' structure.
%{\color{red}Let us recall some of them.}
%
Sigmund~\cite{Sigmund,Sigmund-SPEC} proved that for a compact dynamical system $(X,T)$ with the periodic specification property the measures corresponding to periodic points are dense in $\MT(X)$ (and hence $\MT(X)$ is the Poulsen simplex) and the set of strongly mixing measures is of first category (below we provide more details on Sigmund's work).
%Furthermore, the following sets are complements of sets of the first category in $\MT(X)$:  the set of totally ergodic measures, the set of nonatomic measures, and the set of measures positive on open sets.
Parthasarathy \cite{P} obtained similar results, assuming that $X$ is a product of countably many copies of a complete separable metric space and $T$ is the shift transformation.

Roughly speaking, \emph{periodic specification property} says that given an arbitrary number of arbitrarily long orbit segments, one can find a periodic orbit which stays $\varepsilon$-close to each segment and spends a fixed number (which depends on $\varepsilon$ only) of iterations between the consecutive segments (see Section~\ref{sec:prelimii}).
 This property holds for a number of topological dynamical systems such  as, for example, topologically mixing shifts of finite type, sofic shifts~\cite{DGS}, and  cocyclic shifts~\cite{Kwapisz}, and  topologically mixing continuous maps on the interval~\cite{Blokh83,Buzzi}.
For a smooth dynamical system, specification is closely related to hyperbolicity and holds for any basic set of an axiom A-diffeomorphism \cite{Bowen}.

In other systems the periodic specification property is usually difficult to verify and it fails for example for a general diffeomorphism beyond a uniformly hyperbolic context and for some $\beta$-shifts and $S$-gap shifts. But there still exist numerous extensions of Sigmund's results under weaker hypothesis.
Pfister and Sullivan considered two specification-like notions, coined \emph{approximate product property} \cite{PS1} and \emph{$g$-almost product property} \cite{PS2}. Stated naively, these notions ``allow to make a number of errors'' when an orbit traces a specified family of orbit segments, but the number of these errors decays sufficiently fast as the length of the specified orbit segment growths to infinity.

In this paper, we encompass previous results (\cite{ABC}, \cite{CS},  \cite{Sigmund,Sigmund-SPEC}, \cite{ST}) and extend them to apply to new examples.
Focusing on the periodic points only, we identify two topological properties: \emph{linkability} of a set of periodic points and \emph{closeability} with respect to some set of periodic points.
% (the same idea was independently used in \cite{ABC,CS}).
%
A subset $K\subset\Per(T)$ of periodic points of $T$ is \emph{linkable} if,  briefly stated, the dynamical system $T$ restricted to $K$ satisfies the periodic specification property (see Section~\ref{sec:linkability} for full details).
Given $K\subset\Per(T)$,  a point $x\in X$ is \emph{$K$-closeable} if, roughly speaking, for every $\varepsilon>0$ for infinitely many $n$ one can find a periodic orbit of some point in $K$ with period roughly equal to $n$ which stays $\varepsilon$-close to the initial orbit segment $x,T(x),\ldots,T^{n-1}(x)$ of $x$.
A dynamical system $(X,T)$ is \emph{$K$-closable} if for every ergodic measure there exists a generic $K$-closable point. The latter implies that every ergodic measure is \emph{$K$-approximable}, that is, is weak$\ast$-accumulated by measures supported on periodic orbits of points in $K$ (see Section~\ref{sec:closability} for full details).
Each of these properties allows us to carry over a respective step in the proof of density of ergodic measures. Neither of them implies the other. Together they  imply that every invariant measure has generic points and allow to carry over Sigmund's results on generic properties of measures from $\MT(X)$.

The following is our main result (see Section \ref{sec:prelimii} for precise definition of the concepts involved).

%%K Below we use the notation introduced in Section~\ref{sec:prelimii}.
%\begin{SigThm}\label{thm:Sigmund}
\begin{theorem}\label{thm:main} %Let $\CM(T)$ be the measure center of $T\colon X\to X$.
Let $(X,T)$ be a dynamical system on a Polish metric space.
Assume that $K\subset\Per(T)$ is a linkable subset of periodic points of $T$ and that
% %a dynamical system
$(X,T)$ is $K$-approximable (in particular, if $(X,T)$ is $K$-closeable). Let $\CM(T)$ be the measure center of $T$, that is, the complement of the union of all universally null sets. Then:
\begin{enumerate}
  \item \label{Sig:1} The set $\MTp(K)$ of measures supported on periodic orbits of points in $K$ is dense in $\MT(X)$.
  \item \label{Sig:2} The set $\MTe(X)$ of ergodic measures is residual in $\MT(X)$.
 % \item \label{Sig:2tot} The set $\MTtot(X)$ of totally ergodic measures is residual in $\MT(X)$.
  \item \label{Sig:2a} The set $\MTe(X)$ is arcwise connected provided that $X$ is compact. %and residual %subset of in $\MT(X)$.
  \item \label{Sig:3}  If $X$ is compact, then either $\MT(X)$ is a singleton or $\MT(X)$ is the Poulsen simplex.
  %\item \label{Sig:4} The set $\MTe(X)$ is residual in $\MT(X)$.
  \item \label{Sig:5} The set $\MTe(X)\cap\MTpos(\CM(T))$ of ergodic measures with full support in $\CM(T)$ is residual in $\MT(X)$.
  \item \label{Sig:6} The set $\MTmix(X)$ of strongly mixing measures is of first category in $\MT(X)$.
  \item \label{Sig:7}  A map $T$ restricted to $\CM(T)$ is topologically transitive.
  \item \label{Sig:8} The set $\cup_{n=0}^\infty T^n(K)\subset\Per(T)$ is dense in $\CM(T)$.
  \item \label{Sig:9} If $\MT(X)$ is not a singleton,
  then the set of all non-atomic measures is residual in $\MT(X)$.
  \item \label{Sig:10} For every non-empty, closed, connected $V\subset \MT(X)$ there is a dense set $D\subset\CM(T)$ such that $V_T(x)=V$ for every $x\in D$, where $V_T(x)$ denotes the set of all weak$\ast$ accumulation points of the empirical measure $\frac1n(\delta_x+\ldots+\delta_{T^{n-1}(x)})$. In particular, every invariant measure has a generic point.
  \item \label{Sig:11}The set of points having maximal oscillation is residual in $\CM(T)$.
  \item \label{Sig:12} If $\MT(X)$ is not a singleton, then the set of quasiregular points is of first category.
\end{enumerate}
\end{theorem}
%\end{SigThm}

To prove the above theorem, we essentially follow the proof by Sigmund~\cite{Sigmund,Sigmund-SPEC}, which seems to be natural  and which method was, in fact, applied already many times in the literature.
Our key point is that we identify the properties which are strong enough to reach the desired conclusion, yet flexible enough to apply to various new settings.

We also note that the proof of density of ergodic measures breaks naturally into two independent steps. The first step guarantees density of periodic measures in the set of ergodic measures, while the second one shows that periodic measures are dense in their convex hull. The independence of these steps follows because there is no implication between closeability and linkability. We provide examples which suggest that further generalizations of this approach seem to be hardly possible.
%We end the paper with discussion of some open problems and related examples.

Let us now compare Theorem~\ref{thm:main} with previous results.

First observe that a compact system $(X,T)$ with the periodic specification property is closable with respect to the linkable set $\Per(T)$ (see Propositions~\ref{prop:spec-closeable} and~\ref{prop:spec-link}) and hence Theorem~\ref{thm:main} applies to systems with the periodic specification property as it was essentially shown by Sigmund (compare~\cite{DGS}).
%items~\eqref{Sig:1}--\eqref{Sig:2tot},\eqref{Sig:6}--\eqref{Sig:8},\eqref{Sig:12} are contained in~\cite{Sigmund-SPEC}.
Items~\eqref{Sig:2a} and \eqref{Sig:3} %and Item~\eqref{Sig:9}
 have been shown in particular case of subshifts of finite type~\cite{Sigmund77}
 and of an axiom A diffeomorphism~\cite{Sigmund}, respectively,
 and are new in this generality (though they are almost immediate consequences of~\cite{LOS}).
 %Items~\eqref{Sig:10} extends~\cite{Sigmund-SPEC} by showing the existence of a \emph{dense} set $D$ of such points. Item~\eqref{Sig:11} is new.

 Note that some of our assertions are stated for $T$ restricted to the measure center $\CM(T)$ (see Section~\ref{sec:prelimii}). This is because the behavior of $T$ outside the measure center is negligible from the \emph{measure-theoretic} point of view, even though the measure center may be \emph{topologically} negligible (nowhere dense) in $X$. Thus it is not surprising that properties of the simplex of invariant measures of $(X,T)$ have no influence on the behavior of $T$ outside its measure center $\CM(T)$.
In relation to item~\eqref{Sig:5} let us remark that~\cite{Sigmund,Sigmund-SPEC} for a compact system $(X,T)$ with the periodic specification property show that invariant measures being positive on all nonempty open subsets of $X$ are residual in $\MT(X)$; it follows from the fact that the periodic specification property of $T$ implies density of the set of periodic points in $X$, hence $\CM(T)=X$. In the general case, our hypotheses do not guarantee that such a measure with full support exists.
Moreover, observe that the measure center is always contained in the non-wandering set, but the former can be strictly smaller than the latter.
% (see Proposition~\ref{thm:mix}).
We provide an example to illustrate these facts (see Proposition \ref{thm:mix}).
%, applies to topologically mixing shifts of finite type, sofic shifts~\cite{DGS},  cocyclic shifts~\cite{Kwapisz}, and continuous topologically mixing maps on the interval~\cite{Blokh83,Buzzi}.
%By the Smale Spectral Decomposition Theorem \cite[(6.2)]{Smale}, for a smooth dynamical system, specification is closely linked with hyperbolicity and guaranteed for any basic set of an axiom A-diffeomorphism \cite{Bowen}.
%Let us provide more more details on various extensions of Sigmund's results.

Another remark relates to item~\eqref{Sig:7}. Sigmund~\cite{Sigmund-SPEC} showed that a compact system $(X,T)$ with the periodic specification property is topologically mixing. Under the general hypotheses of Theorem~\ref{thm:main} one can only prove  that $T$ restricted to $\CM(T)$ is topologically transitive. Indeed, observe that the theorem applies to all $S$-gap shifts for which the measure center is the whole shift space and observe that not all $S$-gap shifts are mixing (see Section~\ref{sec:S-gap}).

%See also~\cite{Hof1987,Hof1988,HR}.
%For example, Hofbauer~\cite{Hof1987,Hof1988} and Hofbauer and Raith~\cite{HR} proposed weaker forms of the specification property to study some transitive and not necessarily continuous transformations on the interval.
 In the compact case it is known that specification implies the $g$-almost product property and the latter implies the approximate product property (see \cite{Y}\footnote{Note that there is no standard terminology regarding the variants of the specification property. In particular, the terminology in \cite{Y} is different than ours.}). The converse implications are false.
 %These properties play central role in large deviations theory (see \cite{Comman,PS1}).
 %, and is also a key point in Pfister and Sullivan's approach for computing Billingsley dimension on shift spaces \cite{PS0}.
 Pfister and Sullivan \cite{PS1} proved  that the ergodic measures are \emph{entropy dense} in $\MT(X)$
 for a compact dynamical system $(X,T)$ with  the approximate product property,
  % (see also \cite{EKZ} for an earlier result for compact systems with specification).
that is, for any $\mu\in\MT(X)$, in any neighborhood of $\mu$ in $\MT(X)$ and $\eps>0$ there is an ergodic measure $\nu$ such that
$h_\mu(T)-\eps< h_\nu(T)$. In particular, every measure is the weak$\ast$ limit of a sequence of ergodic measures.
%whose entropies converge towards the entropy of $\mu$.
%, such that the entropy of $\mu$ is the limit of the entropies of the $\mu_n$.
%Entropy density of ergodic measures for compact systems with specification was proved by Eizenberg, Kifer, and Weiss in \cite{EKZ}.
Clearly, in this case, %if ergodic measures are entropy dense for a compact system, then
the simplex of invariant measures is either Poulsen or a singleton.
%On the other hand, ergodic measures can be dense, but not entropy dense  and o
Our properties are strong enough to imply that the ergodic measures are dense, but not necessarily entropy dense (see Propositions~\ref{prop:counter-entropy-gap} and~\ref{prop:counter}). 
%(see \cite[Definition 2.7]{PS1}), which implies density of ergodic measures in $\MT(X)$.

But our approach gives also new results.

We point out that, in comparison to~\cite{Sigmund-SPEC}, our main theorem remains valid for non-necessarily compact dynamical systems, as in the pioneering work of Parthasharaty \cite{P}. They apply, for example, to irreducible topological Markov chains with countable alphabets.

Our approach provides also new insights to $\beta$-shifts and $S$-gap shifts (see  Sections~\ref{sec:S-gap} and~\ref{sec:betashift}). It is known (see \cite{Hof1987,Hof1988} or \cite{Sig76}) that the simplex of invariant measures of a $\beta$-shift is Poulsen, but the existence of generic points  and generic properties of measures (Theorem~\ref{thm:main} items~\eqref{Sig:10} and~\eqref{Sig:3}--\eqref{Sig:6},\eqref{Sig:9},\eqref{Sig:11},\eqref{Sig:12}) are new (cf. \cite{BM}).
These families of shift spaces have rich, yet usually non-uniform structure and methods which apply to them are likely to extend to similar systems. For example these shift spaces are a testing ground for results on intrinsic ergodicity\footnote{A dynamical system $(X,T)$ is \emph{intrinsically ergodic} if it possesses a unique measure of maximal entropy.} developed by Climenhaga and Thompson \cite{CT2012}. A vast majority of both $S$-gap shifts and $\beta$-shifts fail to satisfy the specification property, however they do satisfy our hypotheses. We expect that even more applications will appear in the future. For example we conjecture that Theorem \ref{thm:main} can be applied to systems considered in \cite{BMGP, GM}.

Our \emph{closeability} property is a much weaker version of a property guaranteed by (what it is commonly referred to) the \emph{Closing Lemma}. Such a lemma, roughly speaking, states that \emph{every} piece of orbit \emph{every} time it returns to its initial point sufficiently closely is shadowed by a periodic orbit (see Section~\ref{sec:closability}). Such a result holds, for example, for generic orbits of smooth dynamical systems preserving a hyperbolic measure and, in particular, for uniformly hyperbolic dynamical systems~\cite{KH:95}. It was also shown for some geodesic flows of a (non-necessarily compact) negatively curved manifold \cite{CS}. However, there is no such result for nonsmooth dynamical systems. On the other hand closeability is a purely topological concept which holds for example for some symbolic systems.

Finally, we note that ideas very similar to our concepts of closeability and linkability were independently used in %a somehow less general context in
 the work of Abdenur, Bonatti, and Crovisier \cite{ABC} and Coudene and Schapira \cite{CS} (see also Sun and Tian~\cite{ST}). But to our best knowledge  items~\eqref{Sig:1},\eqref{Sig:5}, and~\eqref{Sig:6}  in Theorem~\ref{thm:main} are new  for $C^1$-generic diffeomorphisms of a compact manifold (the setting of \cite{ABC}) and geodesic flow on a complete connected negatively curved manifold (considered in \cite{CS}). See Sections~\ref{subsec:C1gendif} and~\ref{subsec:flows} for more details.

Observe also that ergodic measures can be dense in the simplex of invariant measures even if the underlying dynamical system has no periodic point at all. This happens, for example, for systems having the approximate product property and hence (entropy) density of ergodic measures \cite{PS1}.
The methods in this paper apply only to systems with many periodic points.
Note that both properties, closeability and linkability, can be adapted to some system without periodic points \cite{Lac:}.

The paper is organized as follows. Section~\ref{sec:prelimii} contains some preliminaries on the general theory of dynamical systems. Section~\ref{sec:3exam} recalls  concepts from symbolic dynamics illustrated by examples. Section~\ref{sec:closability} introduces the concept of closeability and Section~\ref{sec:linkability} introduces linkability.  In Section~\ref{sec:generic} we study the existence of generic points. Theorem~\ref{thm:main} is shown in Section~\ref{sec:mainresult}.
%The main results of this paper are given in Section~\ref{sec:mainresult}.
Section \ref{sec:applications} summarizes and discusses various classes examples to which Theorem \ref{thm:main} applies. Section~\ref{sec:counter} provides examples and counter-examples which put different concepts into relation to each other. Section~\ref{sec:openqu} states some open questions.
%In some places we omit details and definitions as providing them would lengthen the paper without adding anything new. In each case we refer to the literature where the reader can find more details.

\section{Preliminaries}\label{sec:prelimii}

For background material on elementary ergodic theory and topological dynamics  we refer to \cite{DGS, Walters}.

\subsubsection*{Standing assumptions} Throughout this paper, $X$ is a complete separable metric space with metric $\rho$ and $T\colon X\to X$ is a continuous map. Without loss of generality we assume that $T$ is onto.
By a \emph{dynamical system} we mean a pair $(X,T)$ and often identify it with a map $T\colon X\to X$. Let $\N$ denote the set of positive integers and let $\N_0=\N\cup\{0\}$. Recall that a subset of a topological space is of  \emph{first category} if it can be written as countable union of closed nowhere dense sets. It is \emph{residual} if it is a countable intersection of open and dense sets.

\subsubsection*{Choquet theory} Let $K$ be a non-empty metrizable convex compact subset of a locally convex topological vector space. By $\ext K$ we denote the set of extreme points of $K$. We say that $K$ is a \emph{Choquet simplex} if every point of $K$ is the barycenter of a unique probability measure supported on the set of extreme points of $K$ (see \cite[pp. 174--5]{Simon}, an excellent reference for Choquet theory is \cite{Phelps}, see also~\cite{D-book}.

\subsubsection*{The Poulsen simplex}\label{sec:Poulsen}
%Let $K_P$ be a Choquet simplex and let $\ext K_P$ be its set of extreme points.
A \emph{Poulsen simplex} is a Choquet simplex $K_P$ such that
$\overline{\ext K_P}=K_P$ and $\ext K_P\neq K_P$.
We list some of its properties (see~\cite{LOS} for more details).
\begin{itemize}
\item  The Poulsen simplex is unique up to affine homeomorphism: Any two non-trivial metrizable Choquet simplices with dense sets of extreme points are equivalent under an affine homeomorphism.
\item  Any Choquet simplex is affinely homeomorphic to a face of the Poulsen simplex (a face of a simplex $K$ is a closed convex hull of some nonempty subset of $\ext K$).
%\item Any Polish space is homeomorphic with a closed subset of the $\ext K_P$.
\item The Poulsen simplex is homogeneous: Any two faces of $K_P$ that are affinely homeomorphic are
homeomorphic under an affine automorphism of the Poulsen simplex.
%\item The interior (with respect to $K_P$) of any compact subset of $\ext K_P$ is empty.
%\item Every Polish space is homeomorphic to a closed subset of $\ext K_P$.
%\item The set $\ext K_P$ of extreme points of the Poulsen simplex is homogeneus: for any pair of homeomorphic compact subset $K_1$ and $K_2$ of $\ext K_P$
%there exists an autohomeomorphism of $\ext K_P$ which maps $K_1$ onto $K_2$.
\item Let $Q=[0,1]^\infty$ be the Hilbert cube and let $\mathcal{I}=(0,1)^\infty$ be its pseudo-interior. There exists a homeomorphism $\Phi\colon Q\to K_P$ such that $\Phi(\mathcal{I})=\ext K_P$. In particular, $\ext K_P$ is arcwise connected.
%Note that $P$ is homeomorphic to $l_2$.
\end{itemize}

\subsubsection*{Topological dynamics} An \emph{orbit} of a point $x\in X$ is the set $\{T^n(x)\colon n\in\N_0\}$. A point $x\in X$ is \emph{periodic} for $T$ if $T^k(x)=x$ for some $k\in\N$. Any such $k$ is a \emph{period} for $x$ and the least possible period is the \emph{minimal period} of $x$, denoted by $\per(x)$. We write $\Per(T)$ for the set of periodic points of $T$.

We say that $T$ is \emph{transitive} if for every two non-empty open sets $U,V\subset X$ there is $n>0$ such that $U\cap T^{-n}(V)\neq\emptyset$.
A map $T$ is \emph{mixing} if there is $N\in\N$ such that for all $n\ge N$ we have $U\cap T^{-n}(V)\neq\emptyset$. We say that $(X,T)$ is \emph{minimal} if the orbit of every point is dense in $X$.

We say that dynamical system $(Y,S)$ is a \emph{factor} of $(X,T)$ and $(X,T)$ is an \emph{extension} of $(Y,S)$ if there exists a continuous onto map $\Phi\colon X\to Y$ (called a \emph{semiconjugacy}) such that $\Phi\circ T=S\circ \Phi$. A \emph{conjugacy} is a semiconjugacy which is also a homeomorphism. In that case we also say that $(X,T)$ and $(Y,S)$ are \emph{conjugate} or \emph{isomorphic} dynamical systems.

\subsubsection*{Specification} Let $N\in\N_0$ and $I\subset \N_0$ be such that $I=\N_0\cap \bigcup_{i=1}^k [a_i,b_i]$ for some $k\in\N$ and some integers
\[
0= a_1 \le b_1 < a_2\le b_2<\ldots<a_k\le b_k.
\]
A map $\xi\colon I \to X$ is an \emph{$N$-spaced specification} for $T$ if
$a_{j+1}-b_j\ge N$ for $j=1,\ldots,k-1$ and $\xi(t) = T^{t-s}(\xi(s))$ for any $s,t\in\N_0$ such that $a_i\le s\le t\le b_i$ for some $i=1,\ldots,k$.
A specification $\xi\colon I \to X$ is \emph{$\eps$-traced} by a point $z\in X$ if
$\rho(\xi_j,T^j(z))<\eps$ for $j\in I$. We say that $T$ has the \emph{periodic specification property} or simply the \emph{specification property} if for every $\eps>0$ there is an integer $N=N(\eps)$ such that every $N$-spaced specification $\xi\colon I\to X$ is $\eps$-traced by a point $z\in X$ with $T^p(z)=z$, where $p = \max I +N$.

\subsubsection*{Invariant measures}
Let $\M(X)$ denote the set of all Borel probability measures on $X$. In the following, all topological notions refer to the weak$\ast$ topology of $\M(X)$. It is well known that $\M(X)$ equipped with the weak$\ast$ topology is a complete metrizable topological space (see \cite[\S 6.1]{Walters}). The following defines a metric on $\M(X)$
\[
\D(\mu,\nu)=\inf \{\eps>0\colon \mu(A)\le \nu (A^\eps) +\eps,\text{ for every Borel set }A\subset X\},
\]
where $\mu,\nu\in \M(X)$ and $A^\eps=\{x\in X\colon \rho(x,A)<\eps\}$ denotes the $\eps$-neighborhood of $A$  (see \cite{Strassen}). The \emph{support} of a measure $\mu\in\M(X)$, denoted by $\supp\mu$, is the smallest closed set $C\subset X$ such that $\mu(C)=1$.

Let $\MT(X)$ denote the set of all $T$-invariant measures. We write $\MTe(X)$ for the subset of all ergodic measures. We say that $T$ is \emph{uniquely ergodic} if there is only one $T$-invariant measure.

Recall that if $X$ is compact, then $\MT(X)$ is a non-empty Choquet simplex (see \cite[\S6.2]{Walters}). In a non-compact setting $\MT(X)$ can be empty (see \cite{GM}), but is it always convex and closed, and $\MT(X)$ is the closure of the convex hull of $\MTe(X)$. In particular, if $\MT(X)\neq\emptyset$, then $\MTe(X)$ is always a non-empty $G_\delta$-set. Since $\MT(X)$ is a complete
metric space, a subset of $\MT(X)$ is residual if, and only if, it is a dense $G_\delta$.

%A measure $\mu\in\MT(X)$ is \emph{totally ergodic} if it is ergodic with respect to $T^\ell$ for every $\ell=1,2,\ldots$. We denote by $\MTtot(X)$ the set of all totally ergodic measures.
A measure $\mu\in\MT(X)$ is \emph{strongly mixing} if for every two Borel sets $A,B\subset X$ we have $\mu(A\cap T^{-n}B)\to \mu(A)\mu(B)$. We denote by $\MTmix(X)$ the set of all strongly mixing measures.
A measure $\mu$ has \emph{full support} if $\supp\mu=X$.
Given a Borel set $B\subset X$ we write $\MTpos(B)$ for the set of all $\mu\in\MT(X)$ such that  $B\subset \supp\mu$. In particular $\MTpos(X)$ denotes the set of all  measures with full support.

\subsubsection*{Empirical measures and generic points}
Given a point $x\in X$ and $n\ge 1$ we consider the \emph{$n$-th empirical measure of $x$} $\Emp(x,n)$ defined by
\[
	\Emp(x,n)(A)=\frac{1}{n}\#\{0\le j<n\colon T^j(x)\in A\}
\]
for each Borel set $A\subset X$.
%\marginpar{\tiny K: delete the end of this phrase (trivial)? D: OK}
%that is, $\Emp(x,n)(A)$ is the frequency of visits of a sequence $x,T(x),\ldots,T^{n-1}(x)$ in $A$.
Let $V_T(x)$ be the set of all accumulation points of the sequence $(\Emp(x,n))_{n=1}^\infty$. For $x\in X$ the set $V_T(x)$ is a nonempty
closed connected subset of $\MT(X)$ (see \cite[Proposition 3.8]{DGS}).

We say that a point $x\in X$ is \emph{generic} for a measure
$\mu\in\MT(X)$ if $\Emp(x,n)$ converges in the weak$\ast$ topology to $\mu$ as $n\to\infty$, that is,  if $V_T(x)=\{\mu\}$.
%We denote by $G_\mu$ the set of all points which are generic for $\mu$.
A point $x\in X$ is \emph{quasiregular}
with respect to $T$ if it is generic with respect to some
measure $\mu\in\MT(X)$. A point $x\in X$ has \emph{maximal oscillation} if
$V_T(x)=\MT(X)$. %We denote by QT(X) the set of all quasiregular points.

It is a consequence of the Birkhoff ergodic theorem that for every ergodic measure the set of generic points has full measure, in particular every ergodic measure has a generic point.

\subsubsection*{CO-measures}
If $x\in\Per(T)$ with $\per(x)=k$, then the measure $\gamma(x)=\Emp(x,k)$ is invariant for $T$.
We call such a measure a \emph{CO-measure} and by $\MTp(X)$ we denote the set of all CO-measures.
Given a set $K\subset\Per(T)$, denote by $\MTp(K)$ the set of CO-measures supported on orbits of points in $K$.

\subsubsection*{Measure center}
An open set $U\subset X$ is \emph{universally null} if $\mu(U)=0$ for every $\mu\in\MT(X)$. We call  the complement of the union of all universally null sets   the \emph{measure center} of $T$ and denote it by $\CM(T)$. It is the smallest closed set such that $\mu(\CM(T))=1$ for every $\mu\in\MT(X)$. We have $x\in \CM(T)$ if, and only if, for every open set $U\subset X$ containing $x$ there is a measure $\mu\in\MT(X)$ such that $\mu(U)>0$. The measure center $\CM(T)$ is a non-empty, closed, and $T$-invariant set. Thus, one can consider the restriction $T|_{\CM(T)}\colon \CM(T)\to \CM(T)$.

\medskip
%\subsubsection*{A technical lemma}
We end this section with a technical lemma we will frequently use.
Its proof is straightforward.
\begin{lemma}\label{lem:aux}
Let $k\in \N_0$, $m,n\in\N$, $x\in X$, $\mu_1,\mu_2\in\M(X)$, and $\alpha,\beta\in[0,1]$.
\begin{enumerate}
\item  \label{lem:initial-close} If $0\le k< n\le m$, then
  \[
  	\D\big(\Emp(x,m),\Emp(T^k(x),n-k)\big)\le \frac1n(m-n+k).
\]	
%  \item If $T^m(x)=x$ and $0\le k <m$, then $\D(\gamma(x),\Emp(x,nm+k))\le 1/n$. \label{lem:periodic-close}
\item\label{lem:affine-close}
For $j=1,2$ we have $\D(\alpha\mu_1+(1-\alpha)\mu_2,\mu_j)\le \D(\mu_1,\mu_2)$.
\item\label{lem:uniform-close}
If $\nu_1,\nu_2\in\M(X)$ and $\delta=\max\{\D(\mu_j,\nu_j)\colon j=1,2\}$, then
\[
	\D(\alpha\mu_1+(1-\alpha)\mu_2,\beta \nu_1+(1-\beta)\nu_2)
	\le |\alpha-\beta|+\delta.
\]
\item\label{lem:convex-comb}
   If $d_1=\D(\Emp(x,m),\mu_1)$ and $d_2=\D(\Emp(T^{m}(x),n),\mu_2)$, then
  \[
  	\D\Big(\Emp(x,m+n),\frac{m}{m+n}\mu_1+\frac{n}{m+n}\mu_2\Big)
	\le \max\{d_1,d_2\}.
\]
\end{enumerate}
\end{lemma}

\section{Symbolic dynamics and examples}\label{sec:3exam}
The results of this section are included to keep the exposition self-contained. For more details we refer to %the reader to
the book of Lind and Marcus~\cite{LM}.

\subsubsection*{Shift spaces}

Let $\alf$ be an at most countable set of \emph{symbols}.
Usually we will consider $\alf=\alf_r=\{0,1,\ldots,r-1\}$ for some $r\in\N$.
Equip $\Omega=\alf^\N$ with the metric
\[
\rho(x,y)=\left\{
                      \begin{array}{ll}
                        2^{-\min\{k\in\N\colon x_k\neq y_k\}}, &  \hbox{if $x\neq y$;} \\
                        0, & \hbox{if $x  = y$.}
                      \end{array}
                    \right.
\]
Then $(\Omega,\rho)$ is a Polish metric space, which is compact if $\alf$ is finite.
We consider the \emph{shift transformation} or simply \emph{shift} $\sigma\colon\Omega\to\Omega$ given by $\sigma(\omega)_i=\omega_{i+1}$. This is a continuous map and we call the dynamical system $(\Omega,\sigma)$ the \emph{full shift}.
A \emph{shift space} over $\alf$ is any $\sigma$-invariant and closed subset of $\Omega$.

\subsubsection*{Words and languages}
A \emph{word} over $\alf$ is a finite sequence of symbols from $\alf$.
Let $\alf^*$ denote the set of all words over $\alf$.
We write $u^\infty$ to denote the periodic sequence which is the \emph{concatenation} of infinitely many copies of a word $u$.
We say that a finite word $u=u_1\ldots u_n$ \emph{occurs} in a sequence $\omega=(\omega_i)_{i=1}^\infty$ if there is $k\in \N$ such that
$\omega_{k+j-1}=u_j$ for all $j=1,\ldots,n$. Recall that $X\subset\Omega$ is a shift space if, and only if, there is a set $\mathcal{F}$ of words over $\alf$ such that $X$ is the set of all sequences that do not contain any occurrence of a word from $\mathcal{F}$~\cite{LM}. We also call $\mathcal{F}$ the \emph{set of forbidden words} for $X$.
The collection $\lang(X)\subset \alf^*$ of all words occurring in points in a shift space $X$ over $\alf$ is the \emph{language of $X$}.

\subsubsection*{Directed countable graphs}
Let $G=(V,E)$ be a \emph{directed graph} with at most countable set of \emph{vertices} $V$
and at most countable set of \emph{(directed) edges} $E$. We allow multiple edges between vertices.
Each edge $e\in E$ joins an \emph{initial vertex} $\ini(e)\in V$ with a \emph{terminal vertex} $\term(e)\in V$.
A \emph{path} of length $k$ from $u\in V$ to $w\in V$ on a graph $G=(V,E)$ is a finite sequence $\pi=(e_1,e_2,\ldots,e_k)$ of edges of $G$ such that
$\ini(e_1)=u$, $\term(e_k)=w$, and $\ini(e_{i+1})=\term(e_i)$ for $1\le i \le k-1$.
A \emph{closed path} (a \emph{loop}) on $G$ is a path $\pi=(e_1,e_2,\ldots,e_k)$ such that $\ini(e_1)=\term(e_k)$.
We say that $G=(V,E)$ is \emph{irreducible} or \emph{connected} if for any two vertices $u,v\in V$ there is a path on $G$ form $u$ to $v$.

\subsubsection*{Topological Markov chains}

An irreducible \emph{topological Markov chain} given by the connected directed graph $G=(V,E)$ is a shift space over $\alf=E$ which consists of all infinite sequences $(e_1,e_2,\ldots)$ of edges of $G$ such that $\ini(e_{i+1})=\term(e_i)$ for each $i \ge 1$.

\subsubsection*{Labelled graphs and coded systems}
A function $\lab\colon E\to \alf_r$ is a \emph{labeling} of edges of $G$ by symbols from
the alphabet $\alf_r$.  We call the pair $(G,\lab)$ a \emph{labelled graph}. A labelled graph $(G,\lab)$ is \emph{right-resolving}, if, for each vertex $v$ of $G$, the edges starting at $v$ carry different labels.
%, that is, if $e,e'\in E$, $e\neq e'$ and $\ini(e)=\ini(e')$ imply $\lab(e)\neq\lab(e')$.
Given a path $\pi=(e_1,e_2,\ldots,e_k)$ on $G$ we define a map which we also denote by $\lab$ by
$\lab(\pi)=\lab(e_1)\ldots\lab(e_k)\in\words$. The word $\lab(\pi)$ is the \emph{label} of a path $\pi$ induced by $\lab$. It is well known (see \cite{LM}) that the set of labels of paths on an irreducible right-resolving graph $G=(V,E)$ induced by $\lab\colon E\to\alf_r$ defines a language of a shift space $X_{(G,\lab)}\subset \Omega_r$. We say that the shift space $X_{(G,\lab)}$ is \emph{presented} by $(G,\lab)$.
A shift space that can be presented by an irreducible right-resolving labelled graph is called a \emph{coded system} (see \cite{BH}, \cite[pp. 450--2]{LM}). Observe that a label $w=\lab(\pi)$ of any loop $\pi=(e_1,e_2,\ldots,e_k)$ on $G$ gives us a periodic point $x=w^\infty$ of $X_{(G,\lab)}$. We say that $\pi$ \emph{presents} $x$. Let $\Per(G,\lab)\subset \Per(X_{(G,\lab)})$ be the set of all periodic points of $X_{(G,\lab)}$ presented by some loops on $G$. Further information can be found in \cite{BH,FF,Pe}.

Finally we recall that a shift space presented by an irreducible labelled graph is always transitive.

%\subsubsection*{Specification in the context of symbolic dynamics}
%
%A special case, adapted to
%symbolic dynamics, of this definition reads as follows: A shift $X$ has \emph{specification} if
%there is $t\in\N$ such that for any $n\ge1$ and every $w_0,w_1,\ldots, w_n\in \lang(X)$, there
%exist words $v_1,\ldots, v_n\in \lang(X)$  with $|v_i| = t$ for $i=1,\ldots,n$ such that $w_0v_1w_1v_2\ldots v_nw_n\in %\lang(X)$.

%\section{Examples}
\medskip
Among our motivating examples are $\beta$-shifts
and $S$-gap shifts. Generically, they do not have the specification property. They belong to the class of coded systems.

\subsection{$S$-gap shifts}\label{sec:S-gap}

Let $S\subset \N$. The \emph{$S$-gap shift} $X_S$ is the set of all binary sequences such that between any two successive $1$'s,
the number of $0$'s is an integer from $S$. Then $X_S$ is a shift space over $\{0,1\}$ as one can take for the set of forbidden words the collection $\{10^n1\,\colon \,n\notin S \}$. These shift spaces were introduced by Dinaburg in \cite{Dinaburg}. %If $S$ is finite, then $X_S$ is a sofic shift.
Let us order the elements in $S$ and write $S=\{n_1,n_2,\ldots\}$ with $n_i<n_{i+1}$ for $i<\lvert S\rvert$ ($\lvert S\rvert$ may be finite or infinite here). It is easy to prove (see \cite[Example 3.4]{UJ}) that $X_S$ has specification property if, and only if, $\sup_i(n_{i+1}-n_i)<\infty$ and
$\gcd\{n+1\colon n\in S\}=1$. Note that $X_S$ is topologically mixing if, and only if, $\gcd\{n+1\colon n\in S\}=1$. Every $S$-gap shift is coded. To see this let us denote by $\Gamma_S=(V,E_S)$
a directed graph, whose vertices are $v_0,v_1,v_2,\ldots$ and
there is an edge $v_i\to v_j$ in $E_S$ if, and only if, $j=i+1$ or $i\in S$ and $j=0$ (see Figure \ref{fig:gamma}). We label an edge $v_i\to v_j$ with $0$ is $i<j$ and by $1$ otherwise. We denote this labeling of $\Gamma_S$ by $\lab_S$. It is easy to see that this defines a presentation of  $X_S$.

\begin{figure}
\includegraphics[scale=0.995]{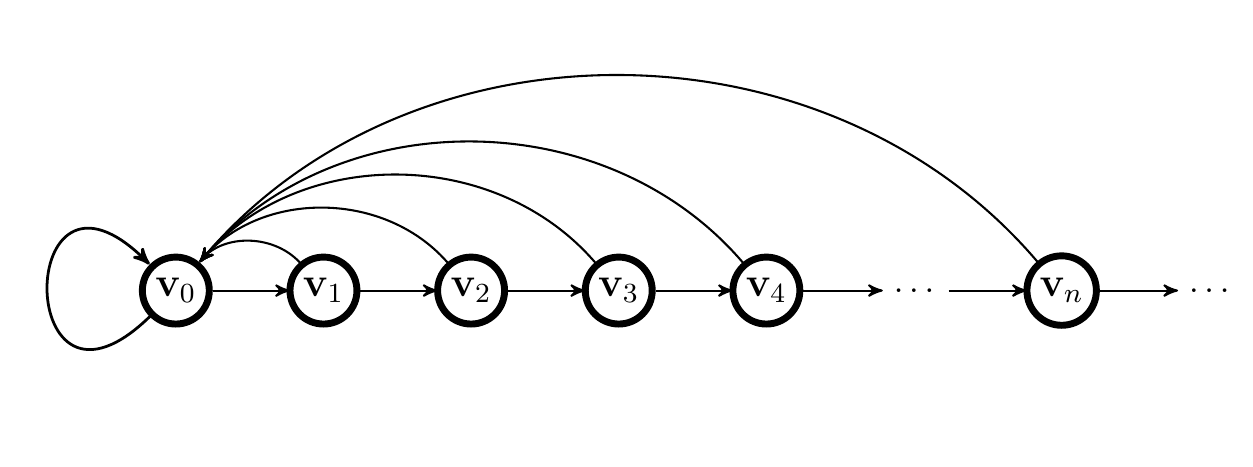}
\caption{Graph $\Gamma_S$ for $S=\N_0$.}
\label{fig:gamma}
\end{figure}

\subsection{$\beta$-shifts}\label{sec:betashift}

Fix $\beta>1$. A \emph{$\beta$-representation} of a number $x\in [0,1]$ is a sequence
%%K $\{b_j\}_{j=1}^\infty$
%\marginpar{\tiny I changed here and below to $(b_j)_{j=1}^\infty$}
$(b_j)_{j=1}^\infty$ with $b_j\in\{0,1,\ldots,\lfloor\beta\rfloor\}$ for every $j\ge 1$ such that
\[
x=\sum_{j=1}^\infty\frac{b_j}{\beta^j}.
\]
This notion was introduced by R\'{e}nyi in \cite{R}.
Given $\beta>1$, a real number $x$ may have many different $\beta$-representations. There is an algorithm called \emph{greedy}, which produces a unique $\beta$-representation~\cite{R}. This special representation is called the \emph{$\beta$-expansion} of $x$. The $j$th ``digit'' of the $\beta$-expansion of $x$ is given by
\[
b_j=\lfloor \beta\cdot T_\beta^{j-1}(x)\rfloor, \quad\text{ where}\quad T_\beta(x)=\beta x -\lfloor\beta x\rfloor =\beta x\mod 1.
\]
Let $d_\beta=(d_j)_{j=1}^\infty$ be the $\beta$-expansion of $1$.
The sequence $d_\beta$ has the following property: if $\sigma$ is the shift on $\N_0^\N$ and $\preceq$ is the lexicographic ordering on $\N_0^\N$, then
\[
\sigma^k(d_\beta)\preceq d_\beta,\quad\text{for all }k\in\N.
\]
By Parry~\cite{Parry} the above condition characterizes sequences $x=(x_i)_{i=1}^\infty$ in $\N_0^\N$ for which there exists a $\beta>1$ such that $x$ is the $\beta$-expansions of $1$.

We say that the $\beta$-expansion $d_\beta$ is \emph{finite} if it ends with an infinite string of $0$'s. In this case we define $\widehat{d}_\beta$ to be
\[
\bar w = d_1d_2\ldots d_{k-1} (d_k-1),\quad \text{and}\quad
\widehat{d}_\beta=\bar{w}^\infty=(d_1d_2\ldots d_{k-1} (d_k-1))^\infty.
\]
Otherwise, if $d_\beta$ does not end with an infinite string of $0$'s, we define $\widehat{d}_\beta=d_\beta$.
In the first case the sequence $\widehat{d}_\beta$ is a $\beta$-representation of $1$, which is a sort of ``improper'' $\beta$-representation, just as $0.999\ldots=1$ in base $10$.

By an abuse of notation we will denote the coordinates of $\widehat{d}_\beta$ by $d_1,d_2,\ldots$. A necessary and sufficient condition for a sequence $b\in\N_0^\N$ to be a $\beta$-expansion of some $x\in [0,1)$ is that
\begin{equation}\label{beta-condition}
\sigma^k(b)\prec \widehat{d}_\beta\quad\text{ for all }k\in\N_0,
\end{equation}
that is, any shift of $b$ is lexicographically strictly less than $\widehat{d}_\beta$. The \emph{$\beta$-shift $X_\beta$} is the closure of the collection of $\beta$-expansion of points in $[0,1)$. By the above results it is easy to see that $X_\beta$ is presented by the graph $\Gamma_\beta=(V,E_\beta,\lab_\beta)$ depicted on Figure \ref{fig:gamma-dois} and described below.
In particular, every $\beta$-shift is a coded system. For more information about $\beta$-shifts  we refer the reader to  \cite{Blanchard, Parry, PS1, Thompson, Th-beta}, among others.

Let $\widehat{d}_\beta= (d_j)_{j=1}^\infty$ %%K\{d_j\}_{j=1}^\infty$
 be the improper $\beta$-expansion of $1$ described above.
By $\Gamma_\beta=(G_\beta,\lab_\beta)$ we denote a directed labelled graph, where $G_\beta$ is a graph whose vertices are $v_0,v_1,v_2,\ldots$, and its edges and their labels are defined by the following rules:
\begin{enumerate}
  \item for each $i=0,1,2,\ldots$ there is an edge $v_i\to v_{i+1}$ labelled by $d_{i+1}$,
  \item if $d_{i+1}>0$, then there are $d_{i+1}$ edges from $v_i$ to $v_0$ labelled by $0,1,\ldots,d_{i+1}-1$.
\end{enumerate}
\begin{figure}
\includegraphics[scale=0.93]{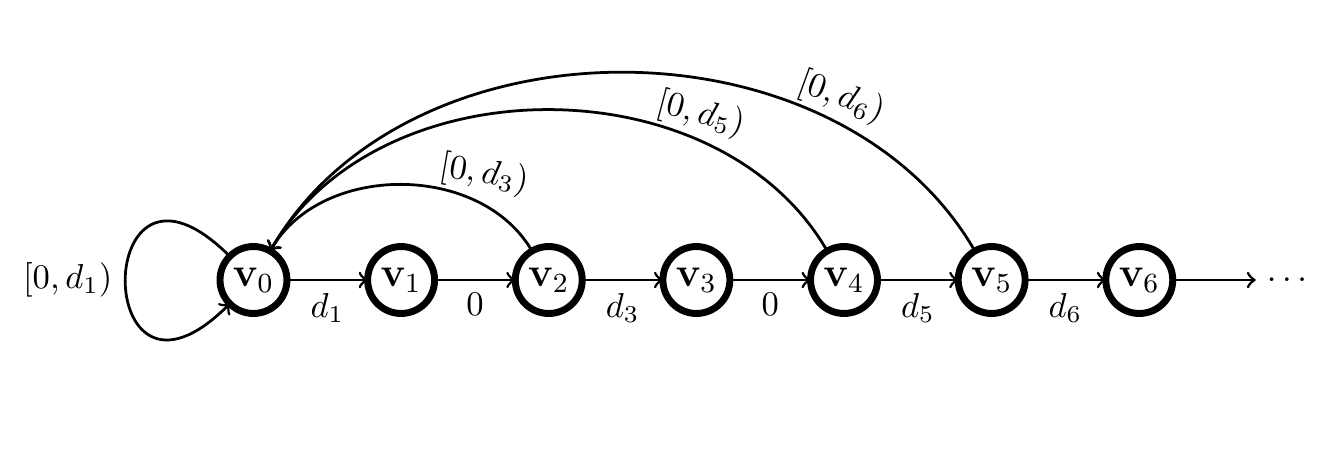}
\caption{Graph $\Gamma_\beta=(V,E_\beta)$ and its labeling $\lab_\beta$ presenting a  $\beta$-shift (by convention: an edge labelled by $[0,d_j)$ means $d_j$ edges labelled by $0,1,\ldots,d_j-1$).}
\label{fig:gamma-dois}
\end{figure}

Note that condition \ref{beta-condition} guarantees that for every word $w$ in a language of $X_\beta$ there is a path in $G_\beta$ starting at $v_0$ and labelled by $w$.

\begin{lemma}[{\cite[\S 1.2]{Buzzi}}]\label{lem:buzzi}
	The set of $\beta>1$ such that the $\beta$-shift $X_\beta$ has the periodic specification property is dense in $(1,\infty)$ but its Lebesgue measure is zero.
\end{lemma}

\section{Closeability and approximability of ergodic measures}\label{sec:closability}

We now introduce the concept of \emph{closeability} for points, invariant measures and dynamical systems. We are inspired by the notion of a \emph{well closeable point} from \cite[Definition 4.3]{ABC}.  Closeability of a dynamical system is a weak version of the celebrated Closing Lemma. Roughly speaking, the Closing Lemma states that every piece of orbit, which comes back close to its initial point is in fact close to an periodic orbit (this holds for example for hyperbolic maps~\cite{KH:95} and for some geodesic flows on the unit tangent bundle of a (non-necessarily compact) negatively curved manifold \cite{CS}).

\begin{definition}\label{def:closeable}
Given a dynamical system $(X,T)$ an $(n,\eps)$-\emph{dynamical ball} (or \emph{Bowen ball}) around a point  $x\in X$ is the set
\[
\db(x,n,\eps)=\{y\in X\colon \rho(T^j(y),T^j(x))<\eps\text{ for }j=0,1,\ldots,n-1\}.
\]
A point $x\in X$ is \emph{closeable with respect to $K\subset \Per(T)$} or simply \emph{$K$-closeable} if for every $\eps>0$ and $N>0$ there exist
integers $p=p(x,\eps,N)$ and $q=q(x,\eps,N)$ such that
there is a point $y\in \db(x,p,\eps)\cap K$ satisfying $T^q(y)=y$ and $N\le p \le q \le (1+\eps) p$.
\end{definition}

Note that a point $x$ is $K$-closeable if for infinitely many $n$'s
the initial segment $x,T(x),\ldots,T^{n-1}(x)$ of the orbit of $x$ can be
 closed in $K$, that is, there exists a
point $y\in K$ with minimal period only slightly greater than $n$, whose orbit ``shadows''
the orbit of $x$ up to the time $n$.
It follows from the next lemma that the $n$-th empirical measure along the orbit of $x$ is close to the ergodic measure concentrated on the orbit of $y$.
\begin{lemma}\label{lem:close}
Let $(X,T)$ be a dynamical system.
Given $x\in X$, $\eps>0$, and $p,q\in\N$ satisfying $p\le q\le (1+\eps)p$, for every $y\in \db(x,p,\eps)$ we have $\D(\Emp(y,q),\Emp(x,p))\le \eps$.
\end{lemma}
\begin{proof}
Take any Borel set $A\subset X$. For every $y\in B(x,p,\varepsilon)$
we have
\[\begin{split}
\Emp(y,q)(A)&=
\frac{1}{q}\left\lvert \{ 0\le j < q \colon T^j(y)\in A\}\right\rvert\\
&\le\frac{1}{p}\left\lvert\{0\le j <p \colon T^j(x)\in A^{\eps}\}\right\rvert
	+\frac{1}{p}\left\lvert\{p\le j < q \colon T^j(y)\in A\}\right\rvert\\
&\le \Emp(x,p)(A^{\eps})+\eps,
\end{split}\]
because $T^j(y)\in A$ implies $T^j(x)\in A^{\eps}$ for $j=0,1,\ldots,p-1$.
Therefore
\[
\Emp(y,q)(A)\le \Emp(x,p) (A^{\eps}) +\eps.
\]
Hence, $\D(\Emp(y,q),\Emp(x,p))\le \eps$.
\end{proof}
We note the following special case of Lemma \ref{lem:close} for further reference.
\begin{remark}\label{closeable-close}
Let $x\in X$ be $K$-closeable. If $y\in K$ and $p,q\in\N$ are chosen for some $N\in\N$ and $\eps>0$ as in Definition \ref{def:closeable}, then
\[
\gamma(y)=\Emp(y,q)\quad\text{and}\quad\D(\gamma(y),\Emp(x,p))<\eps.
\]
\end{remark}

An immediate consequence of definitions given above is the following fact. We do not know whether we can dispense with the uniform continuity assumption.

\begin{lemma}\label{lem:close-conj1}
Let $(X,T)$ be a dynamical system and $(Y,S)$ be its factor through uniformly continuous semiconjugacy $\Phi\colon X\to Y$.
If $x\in X$ is $K$-closeable, then $\Phi(x)$ is $\Phi(K)$-closeable.
\end{lemma}

Now we define closeability for invariant measures and dynamical systems.

\begin{definition}%\label{def:closeable-measure}
A measure $\mu\in\MT(X)$ is \emph{$K$-closeable} if some generic point of $\mu$ is $K$-closeable. A dynamical system $(X,T)$ has the
\emph{$K$-closeability} property if
every ergodic measure $\mu\in\MTe(X)$ is $K$-closeable. A measure $\mu\in\MT(X)$ is \emph{$K$-approximable} if it belongs to the closure
of a set of CO-measures supported on orbits of points from $K$.

We simply say that a point (a measure, a dynamical system) is \emph{closeable} if it is $\Per(T)$-closeable.
\end{definition}

Note that an invariant measure (which may have many generic points)
is $K$-closeable if it has at least one $K$-closeable generic point. On the other hand $K$-closeability for a dynamical system means that all ergodic measures are closeable.  Closeability of a dynamical system $(X,T)$ with respect to a set $K\subset\Per(T)$  implies that all ergodic measures are $K$-approximable, that is, the CO-measures supported on orbits of points from $K$ are dense in $\MTe(X)$.

\begin{lemma}%\label{lem:close-conj2}
Let $(X,T)$ be a dynamical system and $(Y,S)$ be its factor through uniformly continuous semiconjugacy $\Phi\colon X\to Y$.
If $\mu\in \MT(X)$ is $K$-closeable, then its push-forward, $\Phi_\ast(\mu)\in\MS(Y)$, is $\Phi(K)$-closeable.
In particular, if $(X,T)$ is $K$-closeable, then $(Y,S)$ is $\Phi(K)$-closeable.
\end{lemma}
\begin{proof}
Recall that $\Phi_\ast\colon\MT(X)\to\MS(Y)$ is given by
\[
\nu=\Phi_\ast(\mu)\text{ if and only if }\nu(B)=\mu(\Phi^{-1}(B))\text{ for any Borel set $B\subset Y$}.
\]
It is well known that $\Phi_\ast$ is continuous, affine, and  onto. Hence the image of an ergodic measure is ergodic and $\Phi$ maps generic points onto generic points. The rest of the proof follows from Lemma \ref{lem:close-conj1}.
\end{proof}

\begin{proposition}\label{prop:spec-closeable}
	For a dynamical system with the periodic specification property every point is closeable (with respect to $\Per(T)$).
\end{proposition}

\begin{proof}
Let $x\in X$. Take any $N\in\N$ and $\eps>0$. Let $M\in\N$ be provided for this $\eps$ by the specification property. Take $p\in\N$ such that $p>N$ and $M< p\eps$. Set $q=p+M$. By the specification property there is a periodic point $z\in X$ such that $z\in\db(x,p,\eps)$ and $T^{p+M}(z)=z$, which completes the proof.
\end{proof}

The following two results show that the converse is not true, because many $\beta$-shifts and $S$-gap shifts do not have the specification property (compare Lemma~\ref{lem:buzzi} and \cite[Example 3.4]{UJ}).

\begin{proposition}\label{prop:beta-closeable}
Every $\beta$-shift $X_\beta$ is closeable with respect to $\Per(G_\beta,\lab_\beta)$.
\end{proposition}

\begin{proof}
Fix $\beta>1$. Let $\mu$ be an ergodic invariant measure for $X_\beta$. Either $\mu$ is concentrated on the fixed point $0^\infty$, or there exists a symbol $j\in\{1,\ldots,\lfloor\beta\rfloor\}$ such that $\mu([j])>0$. In the former case $\mu$ is clearly closeable as $0^\infty\in\Per(G_\beta,\lab_\beta)$. In the latter case, for every generic point $x=(x_i)_{i=1}^\infty$ of $\mu$ we have $x_i=j$ for infinitely many $i$. For every $k>0$ there exists a path in $G_\beta$ starting at $v_0$ and labelled by $x_1x_2\ldots x_k$. Denote it by $\pi^{(k)}_x=(e^{(k)}_1,\ldots,e^{(k)}_k)$.  Note that for infinitely many $k$ the path $\pi^{(k)}(x)$ ends with an edge labeled by $j>0$. Therefore, either $\pi^{(k)}_x$ is a loop labelled by $w^{(k)}=x_1x_2\ldots x_k$, or we may replace the edge $e^{(k)}_k$ in $\pi^{(k)}_x$ by a backward edge $\bar{e}^{(k)}_k=(v_l\to v_0)$ labelled by $0$ and resulting in a loop $\bar{\pi}^{(k)}_k=(e^{(k)}_1,\ldots,e^{(k)}_{k-1},\bar{e}^{(k)}_k)$ labelled by $w^{(k)}=x_1\ldots x_{k-1}0$. In any case we have found infinitely many $k$ such that $(w^{(k)})^\infty\in\Per(G_\beta,\lab_\beta)$ is a periodic point closing the initial segment of length $k-1$ of the orbit of $x$. It is now easy to check that this implies that $x$ is $\Per(G_\beta,\lab_\beta)$-closeable.
\end{proof}

\begin{proposition}\label{prop:S-gap-closeable}
Every $S$-gap shift $X_S$ is closeable with respect to $\Per(X_S)\setminus\{0^\infty\}$.
\end{proposition}
\begin{proof}
Fix $S\subset\N$. Let $\mu$ be an ergodic invariant measure for $X_S$. Take any point $x=(x_i)_{i=1}^\infty$ generic for $\mu$.  Either $\mu$ is concentrated on the fixed point $0^\infty$, or we have $x_i=1$ for infinitely many $i$. In the former case $S$ must contain a strictly increasing sequence of integers, therefore $(0^k1)^\infty\in X_S$ for infinitely many $k$ and $x$ is clearly closeable. In the latter case, without loss of generality we may assume that $x_1=1$. Now for every  $i$ such that $x_i=1$ there is a periodic point $z_i=(x_1\ldots x_{i-1})^\infty$ in $X_S$. Clearly, the orbit of this point approximates the initial segment of length $i-1$ of the orbit of $x$. It is now easy to check that this implies that $x$ closeable.
\end{proof}

We can now state and prove the main result of this section.
\begin{theorem}\label{thm:closeable}
Let $K\subset \Per(T)$.
If a dynamical system $(X,T)$ is $K$-closeable, then the set
$\MTp(K)$ is dense in $\MTe(X)$.
\end{theorem}
\begin{proof}
Let $\mu\in \MTe(X)$. Take any $\eps>0$. Let $x$ be a closeable generic point for $\mu$. Take $N>0$ be such that
$\D(\mu,\Emp(x,n))<\eps/2$ for every $n\ge N$. Use $K$-closeability for that $N$ and $\eps/2$ to find $p,q\in\N$ with $N\le p\le q\le (1+\eps/2)p$  and a periodic point $z\in B(x,p,\eps/2)\cap K$ such that
$T^q(z)=z$. We have
\[
\D(\gamma(z),\mu)\le \D(\Emp(z,q),\Emp(x,p))+\D(\Emp(x,p),\mu)\le \eps,
\]
because $\D(\Emp(z,q),\Emp(x,p))\le\eps/2$ by Remark~\ref{closeable-close}.
\end{proof}

\section{Linkability}\label{sec:linkability}

We now introduce the \emph{linking property} (or \emph{linkability}).
It can be seen as a specification-like property which applies only to a subset
 of periodic points.

\begin{definition}\label{def:linkability-2}
A set $K\subset\Per(T)$ is \emph{linkable} if for every $y_1,y_2\in K$, $\eps>0$ and $\lambda\in[0,1]$ there exist $p_1,p_2,q_1,q_2\in\N$ and $z\in K$ satisfying the following conditions:
\begin{enumerate}
  \item $T^{q_2}(z)=z$;
  \item $\displaystyle \lambda-\varepsilon\le\frac{p_1}{p_1+p_2}\le\lambda+\varepsilon$;
  \item $p_1\le q_1\le (1+\varepsilon)p_1\quad\text{and}\quad
	z\in B(y_1,p_1,\varepsilon)$;
\item $p_2\le q_2-q_1\le(1+\varepsilon)p_2 \quad\text{and}\quad
	T^{q_1}(z)\in B(y_2,p_2,\varepsilon)$.
\end{enumerate}
\end{definition}
We may (and do) assume that $K$ is $T$-invariant. Furthermore, we show that, given other parameters, one can find $p_1,p_2\in\N$ as above that are divisible by an independently chosen $N\in\N$.
\begin{lemma}
Let $N\in\N$, $\lambda\in[0,1]$ and $\eps>0$.
If $K\subset \Per(T)$ is linkable, $z\in K$, $y_1,y_2\in K$,  then
there are $p_1,p_2,q_1,q_2\in\N$ such that conditions from Definition \ref{def:linkability-2} are satisfied and
$N$ divides $p_j$ for $j=1,2$. In particular, one may assume that  $T^{p_j}(y_j)=y_j$ for $j=1,2$.
\end{lemma}
\begin{proof}
%To see this replace $\lambda$ and $\eps$ by
For any $M\in\N$ one can find $\lambda'\in[0,1]$ and $\eps'>0$ so that
\[
\lambda-\varepsilon<\lambda'-\varepsilon'<\lambda'+\varepsilon'<\lambda+\varepsilon
\]
and there is no rational number with denominator smaller that $M$ between $\lambda'-\varepsilon'$ and $\lambda'+\varepsilon'$. Use Definition \ref{def:linkability-2} to find $p_1,p_2$ and $q_1,q_2$ for $y_1,y_2\in K$, $\lambda'\in[0,1]$ and $\eps'>0$. Note that $p_1+p_2\ge M$. Therefore
if $M$ is sufficiently large, then replacing $p_1$ and $p_2$ by some multiples of $N$ will not affect the other conditions for original $\eps$ and $\lambda$.  We can pick $N$ which is a multiple of minimal periods of $y_1$ and $y_2$. Then $T^{p_j}(y_j)=y_j$ for $j=1,2$.
%Therefore whenever we will apply Definition \ref{def:linkability-2} to find $p_1,p_2\in\N$ we will assume that
\end{proof}

\begin{remark}
The \emph{barycenter property} introduced in \cite[Definition 4.5]{ABC} is implied by the linkability, but it is unclear if the converse is true.
Recall that a set of periodic points $K\subset\Per(T)$ has the \emph{barycenter property} if, for any two points $y_1, y_2 \in K$, any $\lambda \in [0, 1]$ and $\eps > 0$, there exist $x \in K$ and pairwise disjoint sets $I, J \subset \N \cap [0, \per(x))$ such that
%\begin{align}
%\lambda-\eps<\frac{\# I}{\mp(x)}<\lambda+\eps&\text{ and }1-\lambda-\eps<\frac{\# J}{\mp(x)}<1-\lambda+\eps,\\
%\rho(T^m(x),T^m(y_1))<\eps&\text{ for all }m\in I\\\rho(T^m(x),T^m(y_2))<\eps&\text{ for all }m\in J.
%\end{align}
\begin{align}
\lambda-\eps<\frac{\# I}{\per(x)}<\lambda+\eps&\text{ and }1-\lambda-\eps<\frac{\# J}{\per(x)}<1-\lambda+\eps,\\
\rho^T_I(x,y_1)<\eps&\text{ and }\rho^T_J(x,y_2)<\eps,
\end{align}
where $\rho_L^T(z,w)$ denotes a maximum distance between the orbits of $z,w\in X$ over finite $L\subset \N\cup\{0\}$, that is,  $\rho^T_L(z,w)=\max\{\rho(T^m(z),T^m(w)):m\in L\}$.
Note that no further assumption is made about sets $I$ and $J$. They can be arbitrary, while Definition \ref{def:linkability-2} requires that $I$ and $J$ are intervals (contain consecutive integers). Because of this difference the barycenter property can not replace linkability in the proof of existence of generic points.
\end{remark}
Actually, most examples with linkability fulfill the following stronger form.
%%K its strong form.
\begin{definition}\label{def:linkability}
We say that a set $K\subset\Per(T)$ is \emph{strongly linkable} provided that for every $y_1,y_2\in K$ and every $\eps>0$ there exists an integer $N=N(y_1,y_2,\eps)$  such that for any integers $p_1,p_2\ge N$ with
%$T^{p_j}(y_j)=y_j$ for $j=1,2$ there are $z\in K$ and integers $0=q_0\le q_1 \le q_2$ satisfying $T^{q_2}(z)=z$ and $p_j\le q_j-q_{j-1} \le (1+\eps) p_j$ and $T^{q_{j-1}}(z)\in \db(y_j,p_j,\eps)$ for $j=1,2$.
$T^{p_j}(y_j)=y_j$ for $j=1,2$ there are $z\in K$ and integers $q_1 \le q_2$ satisfying $T^{q_2}(z)=z$ and
\[\begin{split}
	&p_1\le q_1\le (1+\varepsilon)p_1\quad\text{and}\quad
	z\in B(y_1,p_1,\varepsilon)\\
	&p_2\le q_2-q_1\le(1+\varepsilon)p_2 \quad\text{and}\quad
	T^{q_1}(z)\in B(y_2,p_2,\varepsilon).\\
\end{split}\]
\end{definition}

\begin{remark}
Strong linkability is in the spirit of non-uniform specification introduced by Climenhaga and Thompson in \cite[Definition 2.1]{CT2012}. The latter applies to shift spaces where
the specification property does not hold globally, but does hold on some collection of subwords. It can be proved that in some cases
the collection of ``good'' words considered in \cite{CT2012} coincides with the collection of blocks defining a (strongly) linkable set of periodic points.
However, it turns out that neither of those properties can be used in place of the other one. Indeed,
%\marginpar{\tiny I added ``Indeed,''}
non-uniform specification   implies intrinsic ergodicity \cite[Theorem C]{CS}, while Proposition \ref{prop:counter-mult-mme} below shows that strong linkability and closeability cannot guarantee that. On the other hand, one can verify that the
shift space from Proposition \ref{thm:independence} satisfies \cite[Definition 2.1]{CT2012}, but the ergodic measures of this system are not dense among all invariant measures.
We refer the reader to \cite{CT2012} and \cite{KOR} for more details on this non-uniform specification.
\end{remark}

%For further reference we note the following consequence of Definition \ref{def:linkability}.

\begin{remark}\label{linkable-close}
Let $K\subset \Per(T)$ be linkable. If $z\in K$ and $p_1,p_2,q_1,q_2\in\N$ are chosen for some $y_1,y_2\in\N$, $\lambda\in[0,1]$ and $\eps>0$ as in Definition \ref{def:linkability} with $T^{p_j}(y_j)=y_j$ for $j=1,2$, then using Lemma \ref{lem:close} we get
\[
	\D(\Emp(z,q_1),\gamma(y_1))
	<\eps\quad\text{and}\quad\D(\Emp(T^{q_1}(z),q_2-q_1),\gamma(y_2))
	<\eps.
\]
It follows from the above and Lemma \ref{lem:aux}\eqref{lem:convex-comb} that $\gamma(z)$
is close to the convex combination $(q_1/q_2)\gamma(y_1)+(1-q_1/q_2)\gamma(y_2)$, more precisely
\begin{equation}\label{inequality}
\D\Big(\Emp(z,q_2),\frac{q_1}{q_2}\gamma(y_1)+\frac{q_2-q_1}{q_2}\gamma(y_2)\Big)
 \le \eps.
\end{equation}
We also have
\[
\frac{p_1}{(1+\eps)(p_1+p_2)}\le\frac{q_1}{q_2}\le \frac{(1+\eps)p_1}{p_1+p_2},
\]
and deducting $p_1/(p_1+p_2)$ we conclude that
\[
\bigg\lvert\frac{q_1}{q_2}-\frac{p_1}{p_1+p_2}\bigg\rvert<\eps.
\]
Now, the last inequality together with \eqref{inequality} and Lemma \ref{lem:aux}\eqref{lem:convex-comb} gives us
\begin{equation}\label{inequality2}
	\D\Big(\Emp(z,q_2),
	\frac{p_1}{p_1+p_2}\gamma(y_1)+\frac{p_2}{p_1+p_2}\gamma(y_2)\Big)
	\le 2\eps.
\end{equation}
By Lemma \ref{lem:aux}\eqref{lem:uniform-close} the inequality $|\frac{p_1}{p_1+p_2}-\lambda|<\eps$ implies
\[
	\D\Big(\lambda \gamma(y_1)+ (1-\lambda)\gamma(y_2),
	\frac{p_1}{p_1+p_2}\gamma(y_1)+\frac{p_2}{p_1+p_2}\gamma(y_2)\Big)
	\le \eps,
\]
therefore
\begin{equation}\label{inequality3}
	\D\Big(\Emp(z,q_2),\lambda \gamma(y_1)+ (1-\lambda)\gamma(y_2)\Big)
	\le 3\eps.
\end{equation}
Hence, $\gamma(z)$ is close to the convex combination $\lambda \gamma(y_1)+ (1-\lambda)\gamma(y_2)$.
\end{remark}

The following lemma is an easy consequence of definitions.

\begin{lemma}%\label{lem:link-conj1}
Let $(X,T)$ be a dynamical system and $(Y,S)$ be its factor through uniformly continuous semiconjugacy $\Phi\colon X\to Y$.
If $K\subset\Per(T)$ is (strongly) linkable, then the same holds for $\Phi(K)\subset\Per(S)$.
\end{lemma}

It is  possible that a proper subset $K$ of $\Per(T)$ is linkable, but $\Per(T)$ is not. A trivial example is when $K$ consists of one point in a system consisting of a finite number (but at least two) of disjoint periodic orbits, less trivial one is presented in Proposition \ref{prop:Dyck} below.

It may also happen that there are two proper linkable subsets $K_1,K_2\subset\Per(T)$ such that $K_1\cap K_2\neq\emptyset$ and $K_1\cup K_2$ is not linkable. A simple example is given by the map $S\colon [-1,1]\to[-1,1]$ defined as
\[
S(x)=
\begin{cases}
T(x),  &x\ge 0,\\
-T(-x),&x<0,
\end{cases}
\]
where $T(x)=1-|2x-1|$ is the \emph{full tent map} on $[0,1]$. Let $K_1=\Per(S)\cap [0,1]$ and $K_2=\Per(S)\cap [-1,0]$. Then $K_1\cap K_2=\{0\}$ and $K_j$ ($j=1,2$) is linkable since $T$ has the periodic specification property, but the set $\Per(S)=K_1\cup K_2$ is not linkable (look at fixed points $2/3$ and $-2/3$). The Dyck shift is a topologically transitive example (even mixing) with a similar behavior (see Proposition \ref{prop:Dyck}).

%As discussed above, only a proper subset of $\Per(T)$ can be linkable.
The following lemma provides a criterion for linkability of the whole system.
Its proof is left to the reader.

\begin{lemma}\label{lem:linkablity-criterion}
If $K\subset\Per(T)$ is (strongly) linkable and every point $x\in\Per(T)$ is $K$-closeable, then $\Per(T)$ is (strongly) linkable.
\end{lemma}

\begin{proposition}\label{prop:spec-link}
If  $T\colon X\to X$ has the periodic specification property, then $\Per(T)$ is strongly linkable.
\end{proposition}
\begin{proof}
Let $y_1,y_2\in \Per(T)$ and $m_1,m_2\in\N$ be their respective minimal periods. Given $\eps>0$ we choose $M=M(\eps)$ by the specification property. Let $n\in \N$ be such that $M<\eps \cdot m_j\cdot n$ for $j=1,2$. We set $N=N(y_1,y_2,\eps)=n\cdot m_1\cdot m_2$. Take any integers $p_1,p_2$ satisfying $p_j\ge N$ and
$T^{p_j}(y_j)=y_j$ for $j=1,2$. Set $q_0=0$,
 $q_1=p_1+M-1$, and $q_2=q_1+p_2+M$. Let $z$ be a point guaranteed by the specification property for the specification
$\xi\colon I\to X$, where $I=\{0,\ldots,p_1-1\}\cup \{q_1,\ldots, q_1+p_2-1\}$  and $\xi(k)=T^{k-q_{j-1}}(y_j)$ for $q_{j-1}\le k \le q_{j-1}+p_j-1$, $j=1,2$. Then  $p_j\le q_j-q_{j-1} \le (1+\eps) p_j$ and $T^{q_{j-1}}(z)\in \db(y_j,p_j,\eps)$ for $j=1,2$.
\end{proof}

\begin{proposition}\label{concatenation-coded}
If $X$ is a coded system presented by a labelled graph $(G,\lab)$, then the set of all periodic points presented by closed paths in $G$ passing through a fixed vertex is strongly linkable.
\end{proposition}
\begin{proof} This is obvious, because one can freely concatenate these closed paths as many times as needed to produce a new closed path with required properties.
\end{proof}

\begin{proposition}\label{prop:beta-link}
Every $\beta$-shift $X_\beta$ is strongly linkable.
%For every $\beta$-shift $X_\beta$ the set $\Per(G_\beta,\lab_\beta)$ is strongly linkable.
\end{proposition}
\begin{proof}
All closed paths of $G_\beta$ pass through $v_0$, hence Proposition \ref{concatenation-coded} applies and the set $\Per(G_\beta,\lab_\beta)$ is strongly linkable. By Proposition \ref{prop:beta-closeable} and Lemma \ref{lem:linkablity-criterion} $X_\beta$ is strongly linkable.
\end{proof}

\begin{proposition}\label{prop:S-gap-link}
Every $S$-gap shift $X_S$ is strongly linkable.
%For every $S$-gap shift $X_S$ the set $\Per(X_S)\setminus\{0^\infty\}$ is strongly linkable.
\end{proposition}
\begin{proof}
By definition of $X_S$ it is easy to see that $\Per(X_S)\setminus\{0^\infty\}=\Per(\Gamma_S,\lab_S)$ and all closed paths in $\Gamma_S$ pass through $v_0$, so we may use Proposition \ref{concatenation-coded} to conclude that $\Per(X_S)\setminus\{0^\infty\}$ is strongly linkable. To finish the proof we apply
Proposition \ref{prop:S-gap-closeable} and Lemma \ref{lem:linkablity-criterion}.
\end{proof}

Recall that the \emph{convex hull}, $\conv(A)$, of a subset $A$ of a locally convex space $\M(X)$ is the smallest convex set containing $A$. Similarly, the \emph{closed convex hull} of $A$,  denoted $\cch(A)$, is the smallest closed convex set containing $A$. Furthermore, we have
\[
\conv(A)=
	\bigcup_{n=1}^\infty
	\Big\{\sum_{j=1}^n\lambda_j\mu_j\colon
		\lambda_j\in[0,1],\,\sum_{j=1}^n \lambda_j=1,\,\mu_j\in A\Big\}
\]
and $\overline{\conv(A)}=\cch(A)$, where $\overline{Y}$ denotes the weak$\ast$ closure of $Y$ (see \cite[Theorem 5.2(i)--(ii)]{Simon}).

Our main result in this section tells us that if $K$ is a linkable subset of $\Per(T)$, then
the set of CO-measures concentrated on points from $K$ is dense in its own closed convex hull.

\begin{theorem}\label{thm:linkable}
If $K\subset\Per(T)$ is linkable, then $\overline{\MTp(K)}=\cch(\MTp(K))$.
\end{theorem}

For the proof we will need the following lemma. %, which says that a set $L$ is dense in $\conv(K)\subset\M(X)$ provided  it is dense in $\{\lambda\mu_1+(1-\lambda)\mu_2\colon \lambda\in [0,1],\,\mu_1,\mu_2\in K\}$.
The statement \emph{a set $A\subset B$ is dense in $B$} means that every point of $B$ is a limit of a sequence in $A$.

\begin{lemma}\label{convex-density}
Let $L\subset K\subset \M(X)$. If $L$ is dense in
$\{\lambda\mu_1+(1-\lambda)\mu_2\colon \lambda\in [0,1],\,\mu_1,\mu_2\in K\}$, then $L$ is dense in $\conv(K)$, hence
\[	
	\overline{L}=\overline{\conv(K)}=\cch(K).
\]
\end{lemma}

\begin{proof}
We assume that for every $\mu_1,\mu_2\in K$, $\eps>0$, and $\lambda\in [0,1]$ there exists $\nu\in L$
such that $\D(\lambda\mu_1+(1-\lambda)\mu_2,\nu)<\eps$. Given
$n\ge 1$ we fix $\mu_1,\ldots,\mu_n\in K$, $\eps>0$, and $\lambda_1,\ldots,\lambda_n\in [0,1]$ such that $\lambda_1+\ldots+\lambda_n=1$.
We search for a measure $\nu\in L$ such that
\[
\D\Big(\sum_{j=1}^n\lambda_j\mu_j,\nu\Big)<\varepsilon.
\]
To this end set $\Lambda_p=\lambda_1+\ldots+\lambda_p$ for $p=1,\ldots,n-1$. We will inductively define measures $\nu_0,\ldots,\nu_{n-1}$ such that $\nu_{n-1}$ is the measure we are looking for.
Let $\nu_0=\mu_1$. Assume that we have already defined measures $\nu_0,\ldots,\nu_{j-1}$ for some $0\le j \le n$. Using our assumption we may find measure $\nu_{j+1}$ such that
\[
\D\Big(\frac{\Lambda_j}{\Lambda_{j+1}}\nu_{j-1}+\frac{\lambda_{j+1}}{\Lambda_{j+1}}\mu_{j+1},\nu_j\Big)<\frac{\eps}{2^j}.
\]
Then
\[
\D\Big(\sum_{j=1}^n\lambda_j\mu_j,\nu_{n-1}\Big)
<\frac{\eps}{2}+\frac{\eps}{4}+\ldots+\frac{\eps}{2^{n-1}}<\eps,
\]
which completes the proof.
\end{proof}

We are now in position to prove Theorem \ref{thm:linkable}.
\begin{proof}[Proof of Theorem \ref{thm:linkable}]
Take $\gamma(y_1),\gamma(y_2)\in\MTp(K)$ supported on orbits of $y_1,y_2\in K$, respectively.
Fix $\lambda\in [0,1]$ and $\eps>0$.
By Lemma \ref{convex-density} it is enough to find $z\in K$ such that
\[
	\D\big(\gamma(z),\lambda \gamma(y_1)+(1-\lambda)\gamma(y_2)\big)
	<\eps.
\]
But the point $z$ provided by linkability is exactly such a point by Remark \ref{linkable-close}.
%Take positive integers $s<t$ such that $\lvert s/t-\lambda\rvert<\eps/3$.
%Let $N=N(y_1,y_2,\eps/3)$ be the constant in the Definition~\ref{def:linkability} of linkability of $K$ for  $y_1,y_2\in K$ and $\eps/3$.
%Choose $p'_1,p'_2\ge N$ with $T^{p'_j}(y_j)=y_j$ for $j=1,2$. Set $p_1=s\,p'_1\,p'_2$ and $p_2=(t-s)\,p'_1\,p'_2$.
%By our choice of parameters and the linking property, there are
%a point $z\in K$ and integers $q_1\le q_2$ satisfying
%Definition \ref{def:linkability}.
%It follows from Remark \ref{linkable-close} that
%\[
%	\D\Big(\Emp(z,q_2),
%		\frac{p_1}{p_1+p_2}\gamma(y_1)+\frac{p_2}{p_1+p_2}\gamma(y_2)\Big)
%	\le \frac{2}{3}\eps.
%\]
%The above inequality combined with $\lvert s/t-\lambda\rvert<\eps/3$ and Lemma \ref{lem:aux}\eqref{lem:uniform-close} imply
%\[
%	\D\big(\gamma(z),\lambda \gamma(y_1)+(1-\lambda)\gamma(y_2)\big)
%	<\eps
%\]
%and the proof is finished.
\end{proof}

%\begin{corollary}
%If the set of periodic points of a dynamical system $(X,T)$ has the linking property, then
%the set of periodic orbit measures is dense in the closed convex hull of itself.
%\end{corollary}

\section{Generic points}\label{sec:generic}

In general, a non-ergodic invariant measure may have no generic points at all. Actually, there exists even an example of a mixing dynamical system with exactly two ergodic measures such that every non-ergodic measure fails to have a generic point (see Proposition \ref{prop:generic-example}). We show that this is not the case when a dynamical system is closeable with respect to a linkable set of periodic points.
\begin{lemma}\label{lem:per-density}
If $T\colon X\to X$ is closeable with respect to a linkable set $K\subset\Per(T)$, then $\cup_{n=0}^\infty T^n(K)\subset\Per(T)$ is dense in $\CM(T)$.
\end{lemma}
\begin{proof}
Take any open set $U\subset X$ such that $U\cap \CM(T)\neq\emptyset$ or, equivalently, such that $\mu(U)>0$ for some $\mu\in\MT(X)$. By Theorems \ref{thm:closeable} and \ref{thm:linkable} there is a sequence $(\gamma(x_n))_{n=1}^\infty\subset\MTp(K)$ such that $\gamma(x_n)\to\mu$ as $n\to\infty$. A well known property of weak$\ast$ convergence (see \cite[p.149, Remark (3)]{Walters})
yields
\[
0<\mu(U)\le\liminf_{n\to\infty}\gamma(x_n)(U).
\]
This means that an orbit of $x_n$ passes through $U$ for all sufficiently big $n$, that is $\cup_{n=0}^\infty T^n(K)\cap U\neq\emptyset$.
%is dense in $\CM(T)$%that is hence \eqref{Sig:8} holds.
\end{proof}

\begin{theorem}\label{thm:generic}
If $T\colon X\to X$ is closeable with respect to a linkable set $K\subset\Per(T)$, then for every $\mu\in\MT(X)$ the set of $\mu$-generic points is dense in $\CM(T)$.
\end{theorem}

\begin{proof} Pick any $\mu\in\MT(X)$ and an open set $U\subset X$ such that $U\cap \CM(T)\neq\emptyset$.
By Theorems \ref{thm:closeable} and \ref{thm:linkable} there is a sequence $(x_n)_{n=1}^\infty\subset K$ such that $\mu_n=\gamma(x_n)\to\mu$ as $n\to\infty$. Using Lemma \ref{lem:per-density} we get that $K$ is dense in $\CM(T)$. Therefore we can find a point $x_0\in K$ and $0<\eps_0<1$ such that the closed ball $\overline{B}(x_0,\eps_0)$ is contained in $U$.
Let $(\eps_n)_{n=1}^\infty$ be any monotonically decreasing sequence of positive reals with $\eps_1<\eps_0$ and $\eps_n\to 0$ as $n\to\infty$.

Without loss of generality we may assume that for each $n\in\N$ we have
\begin{equation}\label{mn}
\D(\mu,\mu_n)\le \eps_n\quad\text{and}\quad\D(\mu_n,\mu_{n+1})\le \eps_n.
\end{equation}
Let $M_n$ denote the minimal period of a point $x_n$ ($n\in\N_0$).

We will inductively construct points $z_0,z_1,\ldots \in K$ and auxiliary sequences $(p^{(n)}_1,p^{(n)}_2)_{n=0}^\infty$ and $(q^{(n)}_1,q^{(n)}_2)_{n=0}^\infty$ in $\N\times\N$.

Set $z_0=x_0$, $p^{(0)}_1=q^{(0)}_1=p^{(0)}_2=M_0$, $q^{(0)}_2=2M_0$. Assume that we have defined $z_0,z_1,\ldots,z_{n-1}\in K$, $(p^{(j)}_1,p^{(j)}_2)_{j=0}^{n-1}$ and $(q^{(j)}_1,q^{(j)}_2)_{j=0}^{n-1}$ for some $n\in\N$.
%Let $N_n=N(z_{n-1},x_n,2^{-n}\eps_n)$ be the constant in the definition of linkability.
Let $\eps<2^{-n}\eps_n$ and $\lambda\in[0,1]$ be such that
\[
\vert \lambda-\frac{p_1^{(n)}}{p_1^{(n)}+p_2^{(n)}}
\vert<\eps
\]
implies
\begin{equation}\label{eqn:p}
\frac{p_1^{(n)}(1+2^{-n}\eps_n)+2^nM_n}{p_2^{(n)}}\le 2^{-n}.
\end{equation}
Let $N_n$ be some multiple of $2^n\cdot M_n$ and $q^{(n-1)}_2$.
For example, let \[
N_n=2^{2n}\cdot M_n\cdot N_n \cdot q^{(n-1)}_2.
\]
We use the linking property to find $z_n\in K$ and integers $p_1^{(n)},p_2^{(n)},q^{(n)}_1 \le q^{(n)}_2$ for $y_1=z_{n-1}$, $y_2=x_n$ and $\lambda,\eps$ as above so that Definition \ref{def:linkability} is fulfilled. In particular, we have $\gamma(z_n)=\Emp(z_n,q^{(n)}_2)$. We assume that $N_n$ divides $p_j^{(n)}$ for $j=1,2$. It implies
$q_2^{(n-1)}$ divides $p_1^{(n)}$ and $M_n$ divides $p_2^{(n)}$, that is
\begin{equation}\label{eqn:multp}
p_1^{(n)}=k_n\cdot q_2^{(n-1)}\text{ for some }k_n\in\N ,\quad  p_2^{(n)}=l_n\cdot M_n\text{ for some }l_n\in\N.
\end{equation}
This completes the induction step and the whole construction.
%Moreover, it follows from our definition that

In what follows we write
\begin{gather*}
r^{(n)}=q_2^{(n)}-q_1^{(n)},\quad y_n=T^{q_1^{(n)}}(z_n),\\%\quad\text{and}\quad
\nu_n=\Emp(y_n,r^{(n)})=\Emp\Big(T^{q_1^{(n)}}(z_n),q_2^{(n)}-q_1^{(n)}\Big).
\end{gather*}
With this notation we have
\begin{subequations}%\label{eqn:db}
\begin{align}
p_1^{(n)}&\le q_1^{(n)} \le (1+2^{-n}\eps_n)p_1^{(n)},\label{db1}\\
p_2^{(n)}&\le r^{(n)} \le (1+2^{-n}\eps_n)p_2^{(n)},\label{db2}\\
z_n&\in\db(z_{n-1},p_1^{(n)},2^{-n}\eps_n),\label{db3}\\
y_n&\in\db(x_{n},p_2^{(n)},2^{-n}\eps_n).\label{db4}
\end{align}
\end{subequations}

Note that by repeated application of \eqref{db3} for any $n=0,1,\ldots$ and $k\in\N$ we have
\[\begin{split}
\max_{j=0,1,\ldots,q_2^{(n)}-1}\rho(T^j(z_n),T^j(z_{n+k}))
&\le \max_{j=0,1,\ldots,q_2^{(n)}-1} \sum_{i=1}^k \rho(T^j(z_{n+i-1}),T^j(z_{n+i}))\\
&\le \frac{1}{2^n} \sum_{i=1}^k \frac{\eps_{n+i}}{2^i}< \frac{\eps_{n+1}}{2^n}.
\end{split}\]
It follows that $(z_n)_{n=0}^\infty$ is a Cauchy sequence, hence it converges to some $\bar{z}\in \overline{B}(z_0,\eps_0)$.
Moreover, $\bar{z}\in\db(z_n,q^{(n)}_2,2^{-n}\eps_{n+1})$. Therefore
\begin{equation*}%\label{D:zn-bz}
\D(\Emp(\bar{z},j),\Emp(z_n,j))\le 2^{-n}\eps_{n+1} \text{ for }j=0,1,\ldots,q_2^{(n)}-1.
\end{equation*}
We claim that $\bar{z}$ is a generic point for $\mu$. To this end we must show that $\Emp(\bar{z},n)\to \mu$ as $n\to\infty$.
Let
\[
\Delta=
\N_0\setminus \Big(\bigcup_{n=0}^\infty \big((p_1^{(n)},q_1^{(n)})\cup (p_2^{(n)},q_2^{(n)})\big)\Big),
\]
where $(a,b)$ denotes an open interval and we agree that $(a,b)=\emptyset$ if $a=b$.
It is easy to see that
\[
d(\Delta)=\lim_{n\to\infty}\frac{\lvert\Delta\cap\{0,1,\ldots,n-1\}\rvert}{n}=1
\]
and $\bar{z}$ is a generic point for $\mu$ if, and only if, $\mu$ is the unique accumulation point of the sequence
$\{\Emp(\bar{z},n)\colon n\in\Delta\}$.
To show that the latter condition holds it suffices to prove that:
\begin{gather}
%\lim_{n\to\infty}\Emp(\bar{z},q_2^{(n)})=
\lim_{n\to\infty}\Emp(z_n,q_2^{(n)})=\lim_{n\to\infty}\gamma(z_n)=\mu,\label{lim:qn}\\
\D(\Emp(z_n,q_2^{(n-1)}+j),\gamma(z_n))\le E(n),\quad j=1,\ldots,p_1^{(n)}-q_2^{(n-1)};\label{lim:qnj}\\
\D(\Emp(z_n,q_1^{(n)}+j),\gamma(z_n))\le E(n),\quad j=1,\ldots,p_2^{(n)};\label{lim:qnjj}
\end{gather}
for some function $E\colon \N\to (0,\infty)$ with $E(n)\to 0$ as $n\to\infty$.

First, we prove \eqref{lim:qn}.
Note that inequalities \eqref{eqn:p} and conditions (\ref{db1}--\ref{db2}) imply that for every $n\in\N$
\begin{equation}\label{eqn:q}
\frac{q_1^{(n)}}{q_2^{(n)}}<\frac{q_1^{(n)}+ 2^{n}M_n}{q_2^{(n)}}\le \frac{p_1^{(n)}(1+2^{-n}\eps_n)+2^nM_n}{p_2^{(n)}}\le 2^{-n}.
\end{equation}
It follows from \eqref{eqn:q} and Lemma \ref{lem:aux}\eqref{lem:initial-close} that
\begin{equation}\label{ineq:1}
\D\big(\Emp(z_n,q_2^{(n)}),\Emp(y_n,r^{(n)})\big)=\D(\gamma(z_n),\nu_n)\le 2^{-n}.
\end{equation}
From \eqref{db2} and \eqref{db4} (see Remark \ref{linkable-close}) we deduce
\begin{equation}\label{ineq:2}
\D\big(\Emp(y_n,r^{(n)}),\Emp(x_n,p_2^{(n)})\big)
=\D\big(\nu_n,\mu_n\big)\le 2^{-n}.
\end{equation}
Then \eqref{ineq:1} combined with \eqref{ineq:2} gives for every $n\in\N$
\begin{equation}\label{lim:final}
\D(\gamma(z_n),\mu_n)\le 2^{-n+1}.
\end{equation}
We have proved that \eqref{lim:qn}  holds.

We proceed to show \eqref{lim:qnj}. Recall that we consider $\Emp(z_n,q_2^{(n-1)}+j)$ for $j=1,2,\ldots,p_1^{(n)}-q^{(n-1)}_2$.
On account of \eqref{db3} for $j=1,2,\ldots,p_1^{(n)}$ we have
\[
\D\big(\Emp(z_n,q_2^{(n-1)}+j),\Emp(z_{n-1},q_2^{(n-1)}+j) \big)\le 2^{-n}\eps_n.
\]
By \eqref{eqn:multp} we may write $q_2^{(n-1)}+j=kq_2^{(n-1)}+s$, where $k\in\{1,\ldots,k_n\}$ and $0\le s< q_2^{(n-1)}$.
If $s<q_1^{(n-1)}$, then \eqref{eqn:q} and Lemma \ref{lem:aux}\eqref{lem:initial-close} give
\begin{multline}\label{eqn:prelim1}
\D\big(\Emp(z_{n-1},kq_2^{(n-1)}+s),\Emp(z_{n-1},kq_2^{(n-1)}) \big)=\\
=\D\big(\Emp(z_{n-1},kq_2^{(n-1)}+s),\gamma(z_{n-1})\big)< \frac{q_1^{(n-1)}}{q_2^{(n-1)}}<2^{-n+1}.
\end{multline}
Suppose $q_1^{(n-1)}\le s \le p_2^{(n-1)}$. Using \eqref{eqn:multp} we may find $t=0,\ldots,l_{n-1}$ and $0\le i <M_{n-1}$ such that $g=kq_2^{(n-1)}+q_1^{(n-1)}$, $m=tM_{n-1}$, and $i=s-q_1^{(n-1)}-m$.
From \eqref{eqn:q} and Lemma \ref{lem:aux}\eqref{lem:initial-close} we obtain
\begin{equation}\label{eqn:prelim2}
\D\big(\Emp(z_{n-1},kq_2^{(n-1)}+s),\Emp(z_{n-1},g+m) \big)\le \frac{i}{q_2^{(n-1)}}<2^{-n+1}.
\end{equation}
Note that $T^g(z_{n-1})=y_{n-1}$. Hence $\Emp(T^g(z_{n-1}),M)=\Emp(y_{n-1},M)$. We can regard $\Emp(z_{n-1},g+m)$ as a convex combination of the measures $\Emp(z_{n-1},g)$ and $\Emp(T^g(z_{n-1}),m)$. Writing $\alpha=g/(g+m)$ leads to
\begin{equation}\label{eqn:affine}
 \Emp(z_{n-1},g+m)= \alpha\Emp(z_{n-1},g)+(1-\alpha)\Emp(y_{n-1},m).
\end{equation}
Since $g=kq_2^{(n-1)}+q_1^{(n-1)}$ and $q_2^{(n-1)}$ is a period of $z_{n-1}$, we conclude from Lemma \ref{lem:aux}\eqref{lem:initial-close} and \eqref{eqn:q} that
\begin{equation}\label{eqn:g}
\D(\Emp(z_{n-1},g),\gamma(z_{n-1}))\le \frac{q_1^{(n-1)}}{q_2^{(n-1)}}\le 2^{-n+1}.
\end{equation}
Definition of $z_{n-1}$ (see \eqref{db4} and Remark \ref{linkable-close}) yields
\begin{equation}\label{eqn:m}
\D(\Emp(y_{n-1},m),\gamma(x_{n-1}))\le \frac{\eps_{n-1}}{2^{n-1}}.
\end{equation}
Combining \eqref{eqn:affine}, \eqref{eqn:g}, \eqref{eqn:m} and Lemma \ref{lem:aux}\eqref{lem:affine-close} we deduce
\begin{equation*}%\label{eqn:affine2}
\D(\Emp(z_{n-1},g+m),\alpha\gamma(z_{n-1})+(1-\alpha)\gamma(x_{n-1}))\le  2^{-n+1}.
\end{equation*}
From this, \eqref{lim:final} and Lemma \ref{lem:aux}\eqref{lem:affine-close} we get
\begin{equation}\label{eqn:final}
\begin{split}
\D\Big(&\Emp(z_{n-1},g+m),\gamma(z_{n-1})\Big)\\
&\le \D\Big(\Emp(z_{n-1},g+m),\alpha\gamma(z_{n-1})+(1-\alpha)\gamma(x_{n-1})\Big)+\\
&\phantom{\le}
+\D\Big(\alpha\gamma(z_{n-1})+(1-\alpha)\gamma(x_{n-1}),\gamma(z_{n-1})\Big)\\
&\le 2^{-n+1} + 2^{-n+2}.
\end{split}
\end{equation}
We can now combine \eqref{eqn:prelim1}, \eqref{eqn:prelim2}, and \eqref{eqn:final} to assert that
$E(n)=2^{-n+3}$ is a function $E\colon\N\to (0,\infty)$ we wanted and this finishes the proof of \eqref{lim:qnj}.

The proof of \eqref{lim:qnjj} is very similar and we will only sketch it.
Recall that we consider $\Emp(z_n,q_1^{(n)}+j)$ for $j=1,2,\ldots,p_2^{(n)}$. Again
the difference between $\Emp(z_n,q_1^{(n)}+j)$ and $\Emp(z_n,q_1^{(n)}+j+i)$ is negligible provided $0\le i< M_n$.
Hence we may consider only $j=\ell\cdot M_n$ for $\ell=1,\ldots, \ell_n$.
By our construction, $\Emp(z_n,q_1^{(n)})$ is close to $\gamma(z_{n-1})$ and $\Emp(y_n,\ell M_n)$ is close to $\gamma(x_n)$.
Moreover, by \eqref{lim:qn} and \eqref{mn} we see that $\D(\gamma(z_{n-1}),\gamma(x_n))$ tends to $0$. But
$\Emp(z_n,q_1^{(n)}+\ell M_n),\gamma(x_n))$ is a convex combination of $\gamma(z_{n-1})$ and $\gamma(x_n)$ and from the above $\D(\Emp(z_n,q_1^{(n)}+\ell M_n),\gamma(x_n))$ must be negligible as $n$ becomes large.
\end{proof}

The following result generalizes \cite[Proposition 21.14]{DGS}. By a \emph{continuum} we mean a \emph{nonempty, compact, and connected set}.

\begin{corollary}\label{cor:V}
If $T\colon X\to X$ is closeable with respect to a linkable set $K\subset\Per(T)$, then for every
continuum $V\subset \MT(X)$ there is a point $x\in X$ such that $V_T(x)=V$.
\end{corollary}
\begin{proof}
Take a continuum $V\subset\MT(X)$. %%K which is nonempty, compact and connected.
As the set $\{\gamma(x)\colon x\in K\}$ is dense in $\MT(X)$, there is a sequence $(\mu_n)_{n\in\N}\subset\{\gamma(x)\colon x\in K\}$ such that $V$ is the set of all its accumulation points and
$\D(\mu_n,\mu_{n+1})\to 0$ as $n\to\infty$.
Then we can repeat the same construction as in the proof of Theorem~\ref{thm:generic}.
\end{proof}

\section{Proof of Theorem~\ref{thm:main}}\label{sec:mainresult}

%Our results are summarized in the following theorem stated for systems which are closeable with respect to a linkable set of periodic points. It generalizes results of Sigmund  describing some generic properties in $\MT(X)$ known previously for the systems with the periodic specification property (see \cite[\S 21]{DGS}).

%We postpone the proof until the end of this section.
%It is a compilation of previous results and some lemmas we are now going to formulate.

In this section we complete the proof of Theorem~\ref{thm:main}. Throughout what follows, we shall assume that $(X,T)$ is a dynamical system, that $K\subset\Per(T)$ is a linkable set, and that $(X,T)$ is $K$-approximable (this holds, in particular, if $(X,T)$ is $K$-closeable).

First we show that transitivity on the measure center is a necessary condition for density of ergodic measures in $\MT(X)$.
\begin{lemma}\label{lem:plus-residual}
The set
$\MTpos(\CM(T))$ is residual in $\MT(X)$.
\end{lemma}
\begin{proof}
Let $\mathcal{U}=(U_n)_{n=1}^\infty$ be a countable base for the topology on $X$. Let $\mathcal{V}$ be a collection
of such elements $V\in\mathcal{U}$ that $V\cap \CM(T)\neq\emptyset$.  As $\mathcal{V}$ is at most countable we may write $\mathcal{V}=(V_n)_{n=1}^\infty$ (we repeat some set infinitely many times in case $\mathcal{V}$ is finite). For each $n\in \N$ there is
a measure $\mu_n\in\MT(X)$ such that $\mu_n(V_n)>0$. Let
$\mu^*=\sum_{n\in\N} 2^{-n-1} \mu_n$. It is easy to see that $\mu^*\in\MT(X)$ and $\supp\mu^*=\CM(T)$.
Let $D_n=\{\mu\in\MT(X)\colon\mu(V_n)=0\}$. Then $D_n$ is clearly closed. Moreover it is nowhere dense since
for every $\mu\in D_n$ and $\eps\in (0,1)$ the invariant measure $\mu_\eps=\eps\mu^*+(1-\eps) \mu $ does not belong to $D_n$
but $\mu_\eps$ approaches $\mu$ as $\eps\to 0$. Therefore, $D=\bigcup_{n\in\N}D_n$ is nowhere dense and hence its complement %in $\MT(X)$, which is
$\{\mu\in\MT(X)\colon\supp\mu=\CM(T)\}$ is residual.
\end{proof}

\begin{proposition}\label{prop:trans}
If $\MTe(X)$ is dense in $\MT(X)$, then $T|_{\CM(T)}$ is tran\-si\-tive.
\end{proposition}
\begin{proof}Ergodic measures are extreme points of $\MT(X)$, hence $\MTe(X)$ is always a $G_\delta$-set in $\MT(X)$.
By the Lemma~\ref{lem:plus-residual}, the subset of measures $\mu\in\MT(X)$ satisfying $\supp\mu=\CM(T)$ is residual in $\MT(X)$. By our assumption $\MTe(X)$ is also residual. Therefore, their intersection $\{\mu\in\MTe(X)\colon\supp\mu=\CM(T)\}$ is residual as well. It is well known that a dynamical system restricted to a support of an ergodic invariant measure is transitive.
\end{proof}

Recall that a measure with \emph{full support} is a measure whose support is $X$.

\begin{corollary}
If  $\MTe(X)$ is dense in $\MT(X)$ and there is an invariant measure with full support, then $T$ is transitive.
\end{corollary}

\begin{proof}Observe that there is a measure with support $X$ if, and only if, the measure center of $T$ is the whole space. Now Proposition~\ref{prop:trans} concludes the proof.
\end{proof}

Observe that the periodic points can be dense in $X$ for a topologically transitive $T\colon X\to X$, but periodic invariant measure can fail to be dense in $\MT(X)$. For example, D\'iaz, Gelfert, and Rams~\cite{DiaGelRam:10} constructed a topologically transitive and partially hyperbolic set $X$ of a local diffeomorphism $T$ which is a homoclinic class of a hyperbolic periodic point (and hence is the closure of hyperbolic periodic points with certain properties, hence $\CM(T|_X)=X$) such that the set of ergodic measures fails to be connected (\cite[Remark 5.2]{DiaGelRam:10}).

Moreover, Weiss \cite{Weiss} constructed an example of a mixing shift space $X$ with dense set of periodic points (hence $\CM(\sigma|_X)=X$) but without an ergodic measure of full support. A similar topologically mixing system, in which the periodic measures are the only ergodic measures is presented in \cite{FKKL}.% This shows that invariant measures of a topologically transitive (even mixing with dense set of periodic points) dynamical system can fail to be dense in the simplex of invariant measures.

Let $\M_T^{co\, k}(K)$ denote the set of CO-measures supported by periodic points from $K$ with minimal periods at most $k\in\N$.

\begin{proposition}\label{prop:k-dense}
If $T\colon X\to X$ is closeable with respect to a linkable set $K\subset\Per(T)$, then
either $\MT(X)$ is a single CO-measure or for each $k\in\N$ the set $\{\gamma(x)\in\MT(X)\colon x\in K\}\setminus \M_T^{co\, k}(X)$ is dense in $\MT(X)$.
\end{proposition}
\begin{proof}
Suppose that that there is a measure $\nu\in\MT(X)$ and its neighborhood $U\subset\MT(X)$ such that $\mathcal{K}^{co\, k}=U\cap \MT(K)\subset \M_T^{co\, k}(X)$. By Lemma~\ref{lem:plus-residual} the set $\MTpos(\CM(T))$ is residual. Hence, there is a measure
$\mu\in U\cap \MTpos(\CM(T))$. Hence $\mathcal{K}^{co\, k}$ is dense in $U$. Therefore, $\mu$ is a measure concentrated on a periodic orbit of length at most $k$. As the support of $\mu$ is $\CM(T)$ we conclude that $\mu$ is the only element of $\MT(X)$.
\end{proof}

%\marginpar{add proof of totally ergodic, property \eqref{Sig:2tot}}

\begin{proof}[Proof of Theorem \ref{thm:main}] As an immediate consequence of Theorems \ref{thm:closeable} and \ref{thm:linkable}
we obtain that \eqref{Sig:1} holds. The condition \eqref{Sig:1} and \cite[Section 3]{LOS} yield
\eqref{Sig:2a} and~\eqref{Sig:3}.
Since the set of extreme points of $\MT(X)$ %a compact convex set
is always a $G_\delta$ set (see \cite{P}), \eqref{Sig:1} implies \eqref{Sig:2}.
Lemma \ref{lem:plus-residual} and \eqref{Sig:2} imply~\eqref{Sig:5}.

A proof of \cite[Proposition 21.13]{DGS} can be rewritten word by word to give
\eqref{Sig:6}. Periodic specification property assumed in \cite[Proposition 21.13]{DGS} is only used to allow an application
of \cite[Proposition 21.13]{DGS}, which is proved above with weaken assumption as Proposition \ref{prop:k-dense}.

We have proved \eqref{Sig:7} as Proposition \ref{prop:trans} and \eqref{Sig:8} in Lemma \ref{lem:per-density} above.
The proof of \eqref{Sig:9} is the same as the proof of \cite[Proposition 21.10]{DGS} with reference to \cite[Proposition 21.13]{DGS} replaced by application of Proposition \ref{prop:k-dense}.
Theorem \eqref{thm:generic} and Corollary \eqref{cor:V} lead to \eqref{Sig:10}.

The reasoning of \cite[Proposition 21.18]{DGS} applies verbatim to \eqref{Sig:11} (one may invoke Theorem \eqref{thm:generic} instead of \cite[Proposition 21.15]{DGS}). Finally it is easy to see that \eqref{Sig:12} is a consequence of \eqref{Sig:11}.
\end{proof}

%
%
%
% !!!ENTROPY???
%
%
\section{Applications}\label{sec:applications}

Theorem \ref{thm:main} can be applied to the following examples:
\begin{itemize}
  \item system with periodic specification property,
  \item $\beta$-shifts,
  \item $S$-gap shifts,
  \item other coded systems.
  \end{itemize}
This we proved already  (see Propositions \ref{prop:spec-closeable} and \ref{prop:spec-link}, Propositions \ref{prop:beta-closeable} and \ref{prop:beta-link}, and Propositions \ref{prop:S-gap-closeable} and \ref{prop:S-gap-link}, respectively). Explicit examples of coded systems which satisfy the assumptions of Theorem \ref{thm:main} are presented in Propositions \ref{prop:counter-mult-mme} and \ref{prop:counter-entropy-gap}. %(\textcolor{red}{see the next section}). \marginpar{\tiny Section~\ref{subsec:flows}?!}

Below we show that Theorem \ref{thm:main} applies also to isolated non-trivial transitive set of a $C^1$-generic
diffeomorphisms. %%K  of Riemannian manifolds.
Another family of examples is provided by transitive flows admitting a local product structure and satisfying the closing lemma (hence to the geodesic flow on a complete connected negatively curved manifold  or irreducible Markov chains defined on a countable alphabet).
Coudene and Schapira give an extensive list of similar examples \cite[page 169]{CS}.

%We prove \eqref{abc} and \eqref{geodesic} below.

\subsection{$C^1$-generic diffeomorphisms}\label{subsec:C1gendif}
Let $M$ be a compact boundaryless manifold. We denote by $\Diff^1(M)$ the space of $C^1$-diffeomorphisms of $M$ endowed
with the usual $C^1$-topology. The next theorem summarizes consequences of Theorem \ref{thm:main} and some results from \cite{ABC}. For completeness we present our results together with some of the consequences of \cite[Theorem 3.5]{ABC}.
Recall that the phrases ``a $C^1$-generic diffeomorphism $f\in\Diff^1(M)$ satisfies ...'' means that
``there exists a residual subset $\mathcal{R}$ of $\Diff^1(M)$ such that every $f \in\mathcal{R}$ satisfies ...''.
Recall that a compact $f$-invariant set $\Lambda\subset M$ is \emph{isolated} if there is an open set $U$ containing $\Lambda$ such that
\[
\Lambda=\bigcup_{n=-\infty}^{+\infty}f^{n}(U).
\]
A transitive set is \emph{trivial} if its consists of a single periodic orbit.

\begin{theorem}\label{thm:abc}
Let $\Lambda$ be an isolated non-trivial transitive set of a $C^1$-generic
diffeomorphism $f \in \Diff^1(M)$. Then:
\begin{enumerate}
  \item The set of periodic measures supported in $\Lambda$ is a dense subset of
the set $\mathcal{M}_f(\Lambda)$ of invariant measures supported in $\Lambda$. Hence,
$\mathcal{M}_f(\Lambda)$ is the Poulsen simplex and the set of ergodic measures is arcwise connected.
  \item %\label{cond:2}
  There is a residual set $\mathcal{N}\subset \mathcal{M}_f(\Lambda)$ such that every measure $\mu\in\mathcal{N}$ satisfies
  \begin{enumerate}
    \item\label{cond:a} $\mu$ is ergodic, but not mixing;
    \item\label{cond:b} $\mu$ has full support, $\supp \mu=\Lambda$;
    \item\label{cond:c} $\mu$ has zero entropy, $h_\mu(f)=0$;
    \item%\label{cond:d}
    for $\mu$-a.e. point $x\in\Lambda$ the Oseledets splitting coincides with %$F_1(x)\oplus\ldots\oplus F_k(x)$, where $F_1\oplus\ldots\oplus F_k$ denotes
        the finest dominated splitting over $\Lambda$;
    \item\label{cond:e} $\mu$ is non-uniformly hyperbolic.
  \end{enumerate}
  \item\label{c1:10} For every non-empty, closed, connected $V\subset \mathcal{M}_f(\Lambda)$ there is a dense set $D\subset\Lambda$ such that $V_T(x)=V$ for every $x\in D$. In particular, every invariant measure $\mu\in\mathcal{M}_f(\Lambda)$ has a generic point.
  \item\label{c1:11} The set of points having maximal oscillation is residual in $\Lambda$.
  \item\label{c1:12} The set of quasiregular points is of first category in $\Lambda$.
\end{enumerate}
\end{theorem}
\begin{proof} For a proof of \eqref{cond:c}--\eqref{cond:e} see \cite{ABC}. Note that the density of periodic (ergodic) measures and \eqref{cond:b} also follows from \cite{ABC} but our proof is new.
To apply Theorem \ref{thm:main} we need to check that linkability and approximability hold $C^1$-generically.
Because isolated  non-trivial transitive sets of a $C^1$-generic diffeomorphisms are relative homoclinic classes
it is enough to show that if $f$ is a $C^1$-generic diffeomorphism, $V\subset M$ is an open set, and $\mathcal{O}$ is a hyperbolic periodic orbit, then the set of periodic points belonging to the relative homoclinic class of $\mathcal{O}$ with respect to $V$ is linkable. The proof follows verbatim the proof of Proposition 4.8 in \cite{ABC}. A careful inspection of \cite[page 23, lines 3--7]{ABC}) and \cite[Section 3.2]{ABCDW} leads to the stronger conclusion of linkability (in \cite{ABC} only barycenter is mentioned). The $\Per(f)$-approximability follows from Ma\~{n}\'{e}'s Ergodic General Density Theorem (see Theorem 3.8 and Proposition 4.1 in \cite{ABC}.
\end{proof}

We stress that items~\eqref{cond:a} and \eqref{c1:10}--\eqref{c1:12} in Theorem \ref{thm:abc} are new. The proof of Theorem \ref{thm:main} does not work if one assumes only barycenter property as defined in \cite{ABC} in place of linkability.

\begin{remark}
Sun and Tian \cite{ST} used yet another barycenter notion and proved \eqref{c1:11} and \eqref{c1:12} of Theorem \ref{thm:abc}. They say that a dynamical system $(X,T)$ has the \emph{barycenter property} if for any two periodic points $p, q \in\Per(T)$, and any $\eps>0$
there exists an integer $N = N(\eps, p, q) > 0$ such that for any two integers $n_1$, $n_2$, there
exists a point $z \in\Per(T)$ such that
$\rho(T^i(z), T^i(p)) < \eps$, for $-n_1\le i\le 0$ and $\rho(T^{i+N}(z), T^i(q)) < \eps)$, for $0 \le i \le n_2$. This variant of the barycenter property suffices to show that for every non-empty, closed, connected $V\subset \MT(X)$ there is a point $x\in X$ such that $V\subset V_T(x)$, but is too weak to yield the reverse inclusion.
\end{remark}

\subsection{Flows} \label{subsec:flows}

Our results remain true if we consider flows (continuous time dynamical systems) instead of transformations. The proofs apply verbatim.
As an example of an application we consider a continuous flow $\varphi\colon\R\times X\to X$ on a complete separable metric space $X$. %We recall that
%Let $X$ be a metric space and let $\phi\colon \R\times X\to X$ be a continuous flow.
The \emph{strong stable} and \emph{$\eps$-strong stable} sets of $x\in X$ are defined by
\begin{align*}
\wss(x)&=\{y\in X\mid \rho(\varphi_t(x),\varphi_t(y))\to 0 \textrm{ as }t\to \infty\},\\
\wsse(x)&=\{y\in \wss(x)\mid \rho(\varphi_t(x),\varphi_t(y))\le \eps \textrm{ for all }t\ge 0\}.
\end{align*}
Similarly, the \emph{strong unstable} and
\emph{$\eps$-strong unstable} sets of $x\in X$ are defined by
\begin{align*}
\wsu(x)&=\{y\in X\mid \rho(\varphi_{-t}(x),\varphi_{-t}(y))\to 0 \textrm{ as }t\to \infty\},\\
\wsue(x)&=\{y\in \wsu(x)\mid \rho(\varphi_{-t}(x),\varphi_{-t}(y))\le \eps \textrm{ for all }t\ge 0\}.
\end{align*}
We say that the flow $\varphi$ has a \emph{product structure} in an open set $V\subset X$ if
for any $\eps>0$ there exists $\delta>0$ such that for all $x,y\in V$ with $\rho(x,y)\le\delta$ there are a real number $t$ with
$|t|\le\eps$ and a point $z\in \wsue(\varphi_t(x))\cap\wsse(y)$.
The flow $\varphi$ \emph{admits a local product structure} if every $v\in X$ has a neighborhood $V$ such that $\phi$ has a product structure in $V$.
(A similar definition can be stated for transformations~\cite{KH:95}. In this case, there is no shift in time, that is, $z\in  \wsue(x)\cap\wsse(y)$.)

We say that \emph{the closing lemma holds} in an open set $V\subset X$ if for every $\eps>0$ there exist $\delta>0$ and $t_0>0$ such that for all $x\in V$ and $t\ge t_0$ with $\rho(x,\varphi_t(x))<\delta$ and $\varphi_t(x)\in V$ there are $x_0\in X$ and $l>0$ with $|l-t|<\eps$, $\varphi_l(x_0)=x_0$, and $\rho(\varphi_s(x),\varphi_s(x_0))<\eps$ for $0<s<\min\{t,l\}$. The flow $\varphi$ \emph{satisfies the closing lemma} if every point in $X$ has a neighborhood in which the closing lemma holds. A discrete time counterpart of the closing lemma is obvious.
\begin{proposition}
If a flow $\varphi\colon\R\times X\to X$ admits a local product structure and satisfies the closing lemma, then it has the linkability property.
\end{proposition}
\begin{proof}
This is a consequence of \cite[Proposition 4.2]{CS} (see also \cite[Lemma 1]{C}).
\end{proof}

 Note that if the flow $\varphi$  satisfies the closing lemma then every recurrent point is $\Per(\varphi)$-closeable. %($\Per(T)$-closeable).
 Observe that the converse need not to be true. To get a discrete-time example it is enough to remove some backward edges from the graph presenting a shift space constructed in Proposition \ref{prop:counter-mult-mme}. We can now summarize the above discussion and state our main theorem about flows.

\begin{proposition}
If a flow $\varphi\colon\R\times X\to X$ admits a local product structure and satisfies the closing lemma, then Theorem \ref{thm:main} applies to $\varphi$.
\end{proposition}

Therefore Theorem \ref{thm:main} generalizes \cite[Theorem 4.2]{CS} and \cite[Theorem 1.1]{CS} (these are results about density and genericity of ergodic invariant measures) and the remaining parts of Theorem \ref{thm:main} extend \cite[Theorem 4.2]{CS} and \cite[Theorem 1.1]{CS}.
In particular, we have the following result which is a consequence of the analysis above and findings of \cite{CS}.

\begin{theorem}
Let $M$ be negatively curved, connected, complete Riemannian manifold. Let $\mathcal{M}$ denote the set of Borel probability measures on the unit tangent bundle $T^1M$ invariant by the geodesic flow. We assume that the non-wandering set $\Omega$ of the geodesic flow $\varphi$ is nonempty.
Then $\varphi$ admits a local product structure and satisfies the closing lemma and Theorem \ref{thm:main} applies to $\varphi$.
\end{theorem}

\section{Counterexamples}\label{sec:counter}

In this section we present examples illustrating differences between our approach and known results.

\subsection{Independence of linkability and closeability}
We prove that closeability and linkability are independent of each other (Theorem \ref{thm:independence}).

\begin{proposition}\label{thm:independence}
   There is a shift space $X'$ which is not $\Per(X')$-closeable, but $\Per(X')$ is strongly linkable.
\end{proposition}

\begin{proof}
Let $V=\{v_0,v_1,v_2,\ldots\}$ and $U=\{u_1,u_2,u_3,\ldots\}$ be two disjoint countably infinite sets.
By $\Gamma'=(V',E')$ we denote a directed graph, whose set of vertices is $V'=V\cup U$ and the set of edges is
\[
	E'=E_f\cup E_b \cup E_s\cup\{v_0\to v_0\},
\]	
where
\begin{align*}
E_f&=\{v_0\to u_1\}\cup \{u_i\to u_{i+1}\colon i\ge 1\}\quad\text{(\emph{forward edges})},\\
E_b&=\{v_{i+1}\to v_i\colon  i\ge 0\}\quad\text{(\emph{backward edges})},\\
E_s&=\{u_{i+1}\to v_i\colon i\ge 0\}\quad\text{(\emph{skew edges})}.
\end{align*}

\begin{figure}
\includegraphics[scale=0.99]{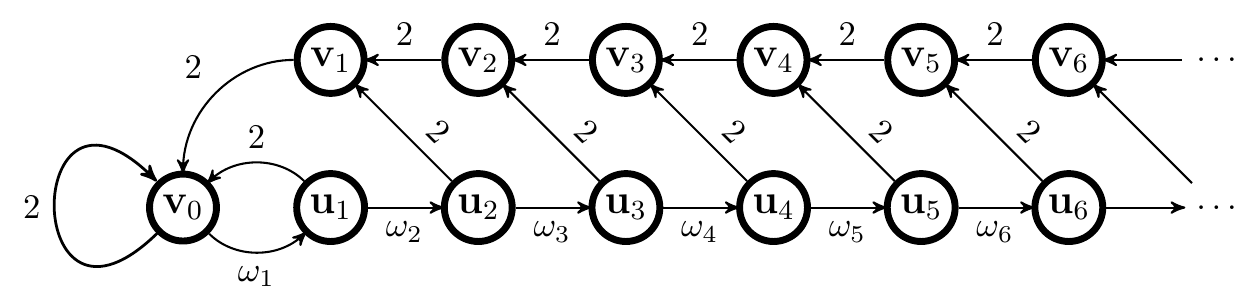}
\caption{Labelled graph presenting a non-closeable system $X'$ with linkable set of periodic points.}
\label{fig:gamma-quattro}
\end{figure}

Let $Y$ be a minimal, uniquely ergodic, %($\M_{\sigma}(Y)=\{\mu_0\}$)
and non-periodic binary shift $Y$ over $\{0,1\}$.
Take $(\omega_i)_{i=1}^\infty\in Y$. %%K$\{\omega_i\}_{i=1}^\infty\in Y$.

Consider the coded system $X'$ presented by the labelled graph $(\Gamma',\Theta')$ depicted on Figure \ref{fig:gamma-quattro}.
In other words $\Theta'\colon E'\to\{0,1,2\}$ is given by
\[
\Theta'(e)=
\begin{cases}\omega_i,& \text{ if $\term(e)=u_i$ for $i\in\N$},\\
2,& \text{ otherwise}.
\end{cases}
\]
It is easy to check that the set of periodic points of this system is linkable (we can freely concatenate
periodic orbits as they all pass through $v_0$). On the other hand $Y$ is a subsystem of $X'$, and the unique invariant measure concentrated on $Y$ is neither CO-measure, nor closeable,
as the measure of the set of all points in $X'$ starting with $2$ is at least $1/2$ for every periodic measure.
\end{proof}

\begin{proposition}%\label{thm:independence-2}
    There is a shift space $X''$ which is $\Per(X'')$-closeable, but $\Per(X'')$ is not linkable.
\end{proposition}

\begin{proof}
We continue using the notation in the proof of the above proposition.

Let $\Gamma''=(V',E'')$ be the directed graph whose set of vertices is $V'=V\cup U$ and the set of edges is
\[
	E''=E_f\cup E_b \cup E_s=E'\setminus\{v_0\to v_0\},
\]	
that is, $\Gamma''$ is $\Gamma'$ with the loop at $v_0$ removed.

Let $X''$ be a coded system presented by the labelled graph $(\Gamma'',\Theta'')$ depicted on Figure~\ref{fig:gamma-tres}.
In other words,
\[
\Theta''(e)=\begin{cases}0,& \text{ for $e\in E_f$},\\
1,& \text{ for $e\in E_b\cup E_s$}.
\end{cases}
\]

Note that the set of all periodic points is not linkable. To see this let
$x_0=0^\infty$ and $x_1=1^\infty$ be fixed points of $X''$ and $\mu_0$ and $\mu_1$ be the point masses concentrated on them.
It is easy to see that the invariant measure $(1/3)\mu_0+(2/3)\mu_1$ cannot be close to any CO-measure.
On the other hand $X''$ is closeable, as every recurrent point is either $x_0$ or $x_1$ or a path labelled
by its initial segment must pass through $v_0$ infinitely many times.
\begin{figure}
\includegraphics[scale=1]{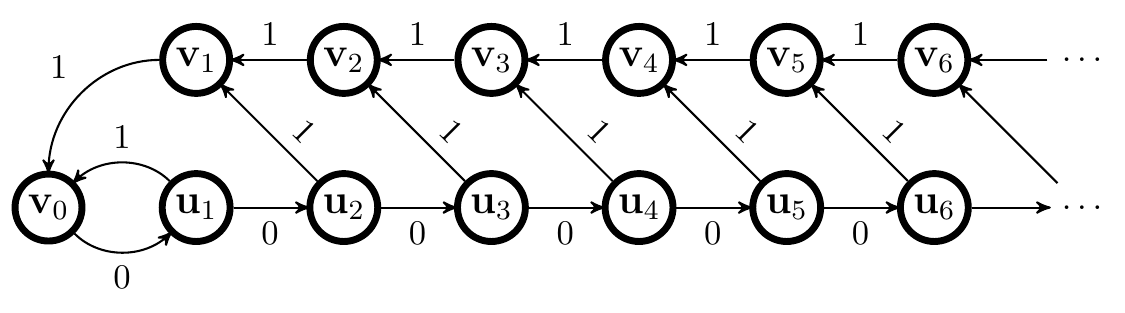}
\caption{Labelled graph presenting a closeable system $X''$ whose periodic points are not linkable.}
\label{fig:gamma-tres}
\end{figure}
\end{proof}

\subsection{Systems with ``small''  measure center}

We show that the measure center $\CM(T)$ can
be a nowhere dense subset of $X$  (thus negligible from the topological point of view), even for mixing systems. For this reason we need to restrict our conclusions to the measure center $\CM(T)$.

\begin{proposition}\label{thm:mix}
Let $X\subset\Omega_r$ be a shift space such that $\CM(X)=X$. There exists a mixing shift space $Y\subset \Omega_{r+1}$ such that $X\subset Y$ and $\CM(Y)=X$.
\end{proposition}
\begin{proof} To specify $Y$ we will describe its language. Let $\mathcal{W}$ be the collection of all words over $\alf_{r+1}$ such that for any word $u$ occurring in $w$ either $u\in\lang(X)$ or if $2^{k-1}+1\le \lvert u\rvert \le 2^{k}$, then the symbol $r$ occurs at less than $k+1$ positions in $u$. It is clear that $\mathcal{W}$ fulfills the assumptions of \cite[Proposition 1.3.4]{LM}, and hence there is a shift space $Y$ with $\mathcal{W}=\lang(Y)$. Moreover, $\words\cap\mathcal{W}=\lang(X)$, in particular $X\subset Y$. Then
\[
\lim_{n\to\infty} \frac{1}{n}\Big\lvert\big\{1\le j \le n : \omega_j=r
\big\}\Big\rvert=0.
\]
for every $\omega\in Y$, hence $\mu([r])=0$ for every $Y$-invariant measure $\mu$. Therefore
the measure center of $Y$ is contained in $Y\cap \Omega_r=X$. It must be then equal $\CM(X)=X$.
Now fix any two cylinders $[u]$ and $[v]$ in $Y$.
For every $k>0$ there is a word $w_k$ of length $k$ in $\lang(X)$ such that $uw_k\in\lang(Y)$.
For all sufficiently large $k$ we have $uw_k(r)v\in \lang(Y)$, and we conclude
that $Y$ is mixing.
\end{proof}

\begin{remark}
Using results and techniques developed in \cite{KKO}  one may replace \emph{mixing} in the statement of Proposition \ref{thm:mix} by any one of
\begin{enumerate}
\item weakly mixing, but not mixing;
\item totally transitive, but not weakly mixing;
\item transitive, but not totally transitive;
\item not transitive.
\end{enumerate}
\end{remark}

%
%
%  !!!EXAMPLE
%
%
\begin{figure}
\includegraphics[scale=0.7]{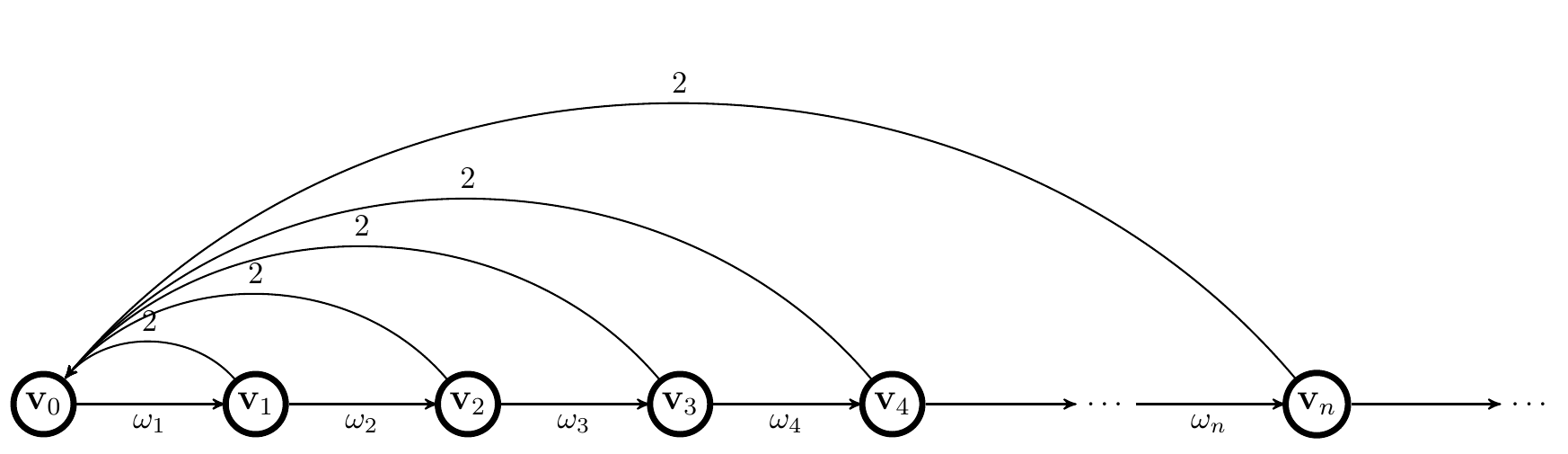}
\caption{Graph $(\Gamma,\Theta)$ used in the proofs of Propositions \ref{prop:counter-mult-mme} and \ref{prop:counter-entropy-gap}.}
\label{fig:gamma-omega}
\end{figure}

\subsection{Closeability, linkability, and intrinsic ergodicity}
The next two examples prove that the closeability property and strong linkability imply neither intrinsic ergodicity, nor entropy density.
Both are inspired by \cite{Pe}.

\begin{proposition}\label{prop:counter-mult-mme}
There is a shift space $X$ which is closeable and linkable, but has multiple measures of maximal entropy.
\end{proposition}
\begin{proof}
Let $Y$ be a strictly ergodic shift over $\{0,1\}$ with topological entropy $h(Y)=\log(1/2(1+\sqrt{5}))$ (logarithm of the golden mean).
We construct a labelled graph presenting a coded system $X$ with the desired properties. To analyze measures of maximal entropy of $X$ we apply the results of Thomsen \cite{Th}. We indicate that the appropriate assumptions are satisfied and we refer the reader to \cite{Th} for more details.

Let $\Gamma=(V,E)$ be a directed graph, whose vertices are labelled $v_0,v_1,v_2,\ldots$,
$E=E_f\cup E_b$ where $E_b=\{v_i\to v_0:i\ge 1\}$ and  $E_f=\{v_{i-1}\to v_i:i\ge 1\}$ (see Figure \ref{fig:gamma-omega}). It is easy to see that
$\Gamma$ is strongly positive recurrent. This means that
\[
\limsup_{n\to\infty}\frac{1}{n}\log l_n< \lim_{n\to\infty}\frac{1}{n}\log r_n,
\]
where $r_n$ is the number of
distinct paths of length $n\ge0$ from $v_0$ to $v_0$ in $\Gamma$ and by $l_n$ we denote the number of paths of length $n$ staring and ending at $v_0$ and not visiting $v_0$ in between (the limit on the right hand side exists since the sequence $r_n$ is clearly submultiplicative, that is, $r_{k+l}\le r_k\cdot r_l$ for all $l,k\ge 1$). A strongly positive recurrent graph is positive recurrent (see \cite{Pe,Th}).

Given
$\omega=(\omega_i)_{i\in\N}\in Y$ we define a labeling $\Theta\colon E\to\{0,1,2\}$ by
\[
\Theta(e)=\begin{cases}
\omega_{i},&\text{if }e=v_{i-1}\to v_{i}\text{ for some }i\in\N,\\
2,&\text{if }e=v_i\to v_0\text{ for some }i\in \N.
\end{cases}
\]
Let $X$ be the coded system presented by $(\Gamma,\Theta)$.
%$(\Gamma_S,\lab^S_\omega)$ the \emph{$S$-gap synchronization of $\omega$} and denote it by $\Syn^{\omega}_S$.

We note that $(\Gamma,\Theta)$ is \emph{follower separated}, that is, if $u,v\in V$ are different vertices then the sets of labels of all paths starting at $u$ and $v$ respectively are different. This holds because $\omega$ is non-periodic. Moreover, $2$ is a \emph{magic word} of $(\Gamma,\Theta)$ because all paths with label $2$ terminate at $v_0$.
This implies that $(\Gamma,\Theta)$ is the Fisher cover of the coded system $X$ and theory from \cite{Th} applies to $X$.
Given a word $w\in\lang(X)$ let $F(w)$ denote the set of all words $u$ such that the concatenation $wu$ is in $\lang(X)$, and $P(w)$ be the set of all words $u$ such that $uw$ is in $\lang(X)$.
The \emph{Markov boundary} of $X$ (see \cite[pp. 1236--37]{Th}) is the shift space $\partial_M X$ with the language
\[
\partial_M \lang(X)=\{w\in\lang(X):\#\{F(xw):x\in P(w)\}=\infty\}.
\]
It is easy to see that $\partial_M X=Y$.
%Moreover, among transitive shifts spaces the Markov boundary is interesting only for non-sofic synchronized systems (see \cite[Theorems 2.3 and 3.1]{Th}).
We recall that the \emph{Gurevi\v{c} entropy} of $\Gamma$ is defined bye $h_\Gamma=-\log R$, where $R$ is the radius of convergence of the power series $\sum r_n z^n$. Straightforward computations show that $h(\partial_M X) = h_\Gamma$. Since $\Gamma$ is positively recurrent, we conclude by \cite[Theorem 7.4(b1)]{Th} that there are two ergodic measures of maximal entropy: one fully supported, and one supported on the Markov boundary. In particular, $X$ is not
intrinsically ergodic.

On the other hand it is clear that $X$ is $\Per(X)$-closeable and $\Per(X)$ is strongly linkable.
\end{proof}

\subsection{Density vs. entropy density of ergodic measures}

The following three properties may or may not be satisfied by $\MT(X)$.
\begin{enumerate}
  \item\label{arc-con} $\MTe(X)$ is arc-connected;
  \item\label{dense} $\MTe(X)$ is dense in $\MT(X)$;
  \item\label{e-den} $\MTe(X)$ is entropy dense $\MT(X)$.
\end{enumerate}
These properties were discussed on Vaughn Climenhaga's blog (\cite{VC}). Clearly, \eqref{e-den}$\Rightarrow$\eqref{dense}$\Rightarrow$\eqref{arc-con}.

The following result (though, the given construction does not have periodic points at all and hence cannot be related to the concepts linkability and closeability) illustrates that ergodic measures being dense and  measures being entropy dense are distinct concepts, that is, \eqref{dense} does not imply \eqref{e-den}\footnote{We are grateful to Tomasz Downarowicz for providing references to~\cite{D-book}.}. It also solves a problem stated by Climenhaga \cite{VC}. Another counterexample is given in Proposition \ref{prop:counter-entropy-gap}. Note that it also proves that our assumptions of closeability and linkability are strictly weaker than the $g$-almost product property of Pfister and Sullivan \cite{PS1}.

\begin{proposition}\label{prop:counter}
There exists a minimal dynamical system $(X,T)$ (a Toeplitz subshift) such that $\MT(X)$ is the Poulsen simplex and there is only one
ergodic measure with positive entropy (in particular, $T$ is not entropy dense).
\end{proposition}

\begin{proof}
Let $K_P$ be the Poulsen simplex. %(see Downarowicz book \cite{Dbook} Definition A.2.11).
Fix an extreme point $\mu\in\ext K_P$. Let $\delta_\mu$ denote the probability measure concentrated on $\{\mu\}$ and let $\eta_\mu$ denote the characteristic function of $\{\mu\}\subset\ext K_P$. Given a point $\nu\in K_P$, we denote by $\xi^\nu$ the unique probability measure concentrated on $\ext K_P$ such that $\nu$ is the barycenter of $\xi^\nu$.
We define a function $\varphi \colon K_P\to \R$ by
\[
	\varphi(\nu)=\int_{\ext K_P} \eta_\mu(\lambda)\xi^\nu(d\lambda)=\xi^\nu(\{\mu\})\,.
\]	
This function is the harmonic prolongation of $\eta_\mu$ \cite[Definition A.2.18]{D-book}.
Since the characteristic function of a closed set is upper semicontinuous, we conclude by~\cite[Fact A.2.20]{D-book} that the same is true for $\varphi$.
By \cite[Fact A.2.10]{D-book} every upper semicontinuous function of a Choquet simplex is affine if and only if it is harmonic. Therefore, $\varphi$ is bounded, affine, nonnegative, and upper semicontinuous on $K_P$ and $\varphi|_{\ext K_P}=\eta_\mu$.

By~\cite[Theorem 1]{DS}, there exists a minimal Toeplitz subshift $(X,T)$ and an affine
(onto) homeomorphism $\psi^\ast\colon K_P \to \MT(X)$, such that for every $\mu \in K_P$,
$\varphi(\mu) =h(\psi^\ast(\mu))$, where $h$ denotes the entropy function that associates to every $T$-invariant probability measure its Kolmogorov-Sinai entropy~\cite[\S 4]{Walters}. This proves the proposition.
\end{proof}

\begin{proposition}\label{prop:counter-entropy-gap}
There is a shift space $X$ which is closeable and linkable, but ergodic measures are not entropy dense in $\mathcal{M}_\sigma(X)$.
\end{proposition}
\begin{proof}
Let $Y$ be a strictly ergodic shift over $\{0,1\}$ with topological entropy $h(Y)>\log(1/2(1+\sqrt{5}))$ (logarithm of the golden mean). We repeat the construction from the proof of Proposition \ref{prop:counter-mult-mme}, but this time we conclude that $h(\partial_M X) > h_\Gamma$. Then we apply \cite[Theorem 7.4(c)]{Th} and obtain that the unique measure of maximal entropy of $X$ is the one supported on the Markov boundary. Furthermore, the entropy of every other ergodic invariant measure of $X$ is not greater than $h_\Gamma$. In particular, ergodic measures are not entropy dense in $\mathcal{M}_\sigma(X)$.
It is obvious that $X$ is $\Per(X)$-closeable and $\Per(X)$ is strongly linkable.
\end{proof}

The last example is also inspired by \cite{VC}, where it is used to show that \eqref{arc-con} does not imply \eqref{dense}.

\begin{proposition}\label{prop:Dyck}
There is a mixing shift space $X$ such that $\MT(X)$ is arcwise connected, but not Poulsen.
Moreover, there are sets $K',K''\subset \Per(X)$ such that $K'$, $K''$, $K'\cap K''$ are infinite and linkable, every measure in $\MT(X)$ is closeable with respect to $K'\cup K''$, but $K'\cup K''$ is not linkable.
%Let $(X,T)$ be the Dyck shift. Then $\MT(X)$ is arcwise connected, but not Poulsen. The sets $K'$, $K''$ and $K$ are infinite and linkable, each measure $\mu$ is closeable with respect to $K'\cup K''$, but $K'\cup K''$ is not linkable.
\end{proposition}
%\begin{proof}
The proof of this proposition follows from the inspection of the definition of the Dyck shift and uses the same techniques as presented above.
We leave the details to the reader.
%\end{proof}
The Dyck shift space was introduced  by Krieger \cite{Krieger} as a counter-example for a conjecture of B. Weiss. Krieger proved that this shift has exactly two measures of maximal entropy, both of which are Bernoulli. We sketch its construction.

The easiest way to describe the Dyck shift is in terms of brackets. We take as alphabet the set
$\alf_D=\{[\,,]\,,(\,,)\,\}$ that consists of two pairs of matching brackets. The language of the Dyck shift comprises
all words over $\alf_D$ in which the brackets are properly nested, in other words, opened and closed in the right order.
To make it precise we define %a monoid on the set
$\Sigma=\{[\,,]\,,(\,,)\,,\mathbf{1},\mathbf{0}\}$. %The set  $\Sigma^*$ equipped with the concatenation operation is a free monoid over $\Sigma$. Consider the following relations
Let $\cdot\colon\Sigma\times\Sigma\to\Sigma$ be defined by
\begin{align*}
[\,\cdot \,] = (\,\cdot\,)&=\mathbf{1},& [\,\cdot\,)=(\,\cdot\,]&=\mathbf{0},\\
a\cdot\mathbf{1}=\mathbf{1}\cdot a&=a,& a\cdot\mathbf{0}=\mathbf{0}\cdot a&=\mathbf{0},
\end{align*}
where $a\in\Sigma$.

%Define an equivalence relation $\sim$ on $\Sigma^*$ by
%\[
%u\equiv v
%\]
%Then the quotient space is a monoid $\mathcal{D}$ with induced operation $\cdot$.
Let $\emptyw$ denote the empty word. The \emph{reduction} map $\red$ from $\alf_D^*$ to $\Sigma^*$ is given by
\begin{align*}
\red(\emptyw)&=\mathbf{1},&\red(w_1\ldots w_n)&=w_1\cdot\ldots\cdot w_n.
\end{align*}
The language of the Dyck shift $X_D$ is the set $\lang_D=\{w\in\alf_D^*\colon \red(w)\neq\mathbf{0}\}$.
The Dyck shift is coded. It is presented by the labelled graph $(G_D,\lab_D)$, where $G_D=(V_D,E_D)$ is given by $V_D=\{v_1,v_2,\ldots\}$
and
\begin{gather*}
E_D=\{v_i\to v_{2i}\colon i\in\N\}\cup\{v_i\to v_{2i+1}\colon i\in\N\}\cup\{v_i\to v_{\lfloor i/2 \rfloor}\colon i\in\N\}\cup\\
\cup \{e',e''\colon \ini(e')=\ini(e'')=\term(e')=\term(e'')=v_1,\,e'\neq e''\}
\end{gather*}
with the labelling $\lab_D$ depicted on Figure \ref{fig:Dyck}.

\begin{figure}
\includegraphics[scale=0.6]{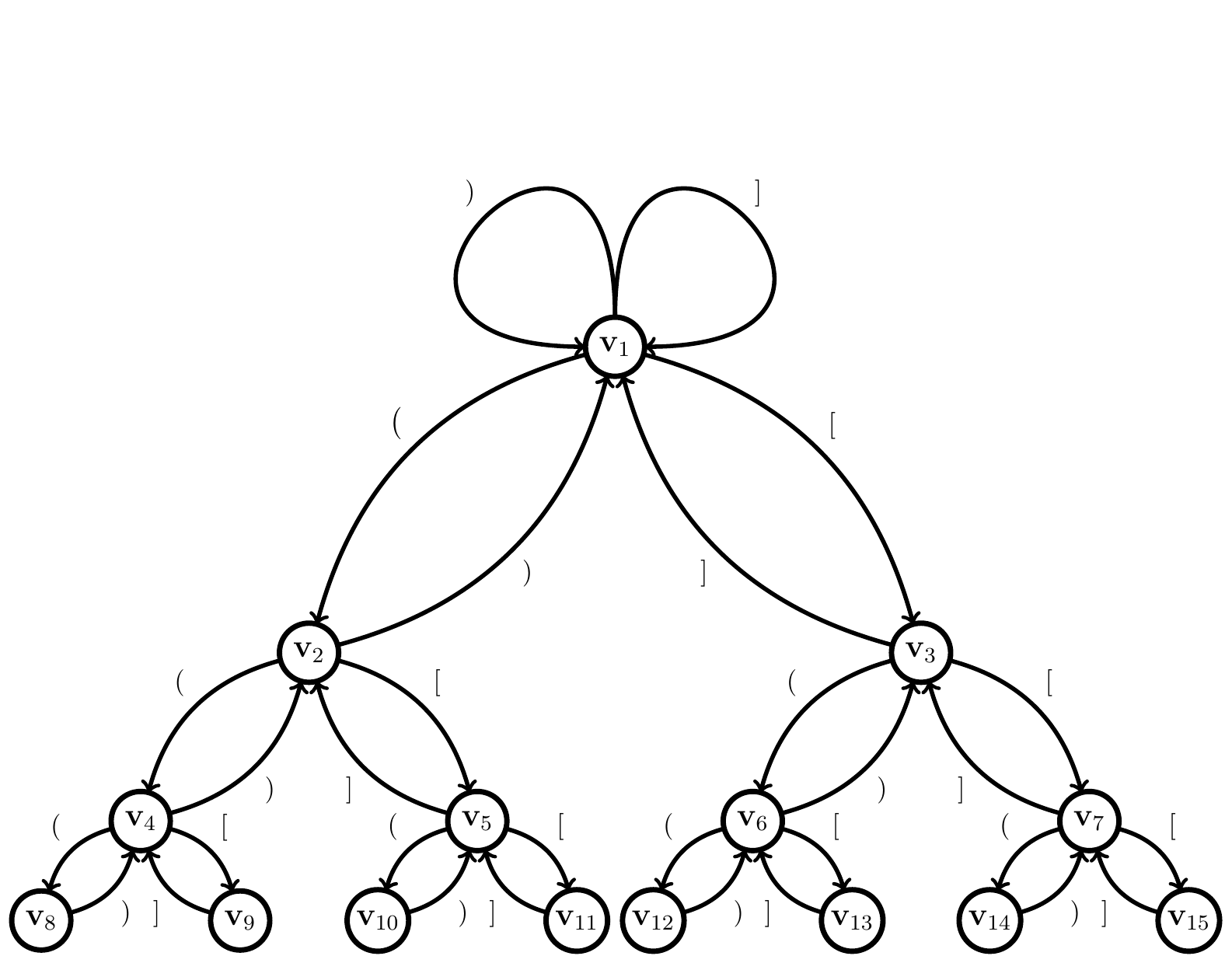}
\caption{A part of a labeled graph $(G_D,\Theta_D)$ presenting the Dyck shift.}
\label{fig:Dyck}
\end{figure}

The graph $G_D$ is build on an infinite complete binary tree (a tree with a countably infinite number of levels, in which every vertex has two children, so that there are $2^d$ vertices at level $d$ and the set of all vertices is countably infinite).
The vertex at the first level is called a \emph{root}. There are two loops $e'$ and $e''$ at the root labelled by $)$ and $]$. Each vertex has two edges leading to two vertices one level down and labelled by $($ and $[$, respectively. Each edge leading down is accompanied by an opposite edge labelled by matching bracket.

\begin{proof}[Sketch of the proof of Proposition \ref{prop:Dyck}.]
Let $\Per_D(\sigma)$ be the set of periodic points of the Dyck shift.
If $x\in \Per_D(\sigma)$ then there exits the shortest word $w\in\lang_D$ such that $x=w^\infty$.
Define
\begin{align*}
K'&=\{w^\infty\in\Per_D(\sigma)\colon\red(w)\in\{\mathbf{1}\,,[\,,(\,\}^\ast\,\},\\
K''&=\{w^\infty\in\Per_D(\sigma)\colon\red(w)\in\{\mathbf{1}\,,)\,,]\,\}^\ast\,\}.
\end{align*}
Then
\[
%K=
K'\cap K''=\{w^\infty\in\Per_D(\sigma)\colon\red(w)=\mathbf{1}\}.
\]
It can be shown that the (infinite) sets $K',K'',K'\cap K''\subset \Per_D(\sigma)$ are strongly linkable and  every measure in $\M_{\sigma}(X_D)$ is closeable with respect to $K'\cup K''$, but $K'\cup K''$ is not linkable.
\end{proof}
%For example, $()[]$ is a legal word, as is $( ( ( ) [ ] [$, but $(
%[ [ )$ is illegal because the ( bracket cannot be closed before the [ brackets are.

\subsection{Measures without generic points}
We prove existence of a shift space such that for every point and every continuous function the Birkhoff averages converge, but not necessarily to a unique value.

\begin{proposition}\label{prop:generic-example}
There exists a topologically mixing shift space $X$ with exactly two ergodic measures such that every point $x\in X$ is a generic point for one of them. In particular, non-trivial convex combination of these ergodic measures has no generic points.
\end{proposition}
\begin{proof}
Let $Y$ be a strictly ergodic shift over $\{0,1\}$. Assume that $Y\neq \{0^\infty\}$. Note that it implies that there is a $K>0$ such that $0^k\notin \lang(Y)$ for all $k\ge K$. We define a language $L$ for a new shift space $X$.
We say that a word $w$ is \emph{allowed} if %and only if  $w\in L$.
\begin{enumerate}
  \item\label{lang-1} $w=0^j$ for some $j>0$, or
  \item\label{lang-2} $w\in \lang(Y)$, or
  \item\label{lang-3} $w=0^a u 0^b$ for some allowed word $u$ and $a,b\ge 0$, or
  \item\label{lang-4} $u=v0^kw$ for allowed words $v$ and $w$ and $k\ge K$ such that
\[
\#\{i:v_i=1\}+\#\{i:w_i=1\}\le \log_2(k+|w|+|v| ).
\]
\end{enumerate}
Formally, the set of allowed words $L$ is defined by induction on the length of the word $w$. The details are left to the reader.
Note that in order to glue an allowed word $v$ with another allowed word $w$, one has to count $1$'s occurring in $v$ and $w$ and glue $v$ and $w$ by inserting a long block of $0$'s in between. Suppose that there are $n$ occurrences of $1$ in $u$ and $v$ altogether. There should be enough $0$'s glued in so that the resulting word $v0^kw$ has length greater than or equal to $2^n$.

The set $L$ of allowed words is a language of a shift space $X$, because $L$ is nonempty (it contains $\lang(Y)$ and $\lang(\{0^\infty\})$ by \eqref{lang-1} and \eqref{lang-2}. By \eqref{lang-3} one can always extend an allowed word by adding $0$ at the beginning or at the end.

It follows from \eqref{lang-4} that any two allowed words can be joined by long enough sequence of $0$'s, hence the shift space $X$ is topologically mixing.

There are two minimal sets: $Y$ and $\{0^\infty\}$. Hence there are two ergodic measures: the first is the unique invariant measure $\mu_M$ on $M$, the second is the invariant measure $\mu_0$ concentrated on the fixed point $0^\infty$.

We claim that every point $x\in X$ is generic either for $\mu_M$ or for $\mu_0$.

If $x=u0^\infty$ for some allowed word $u$, or $x=vy$ for some allowed word $v$ and $y\in M$ then we are done.
If neither of the above hold, then $x$ is of the form
\[
x=x_1x_2x_3\ldots=0^{l(0)}u_1 0^{l(1)}u_20^{l(2)}u_3\ldots
\]
where $l(0)\ge 0$, $u_j\in \lang(Y)$ and $l(j)\ge K$ for each $j\ge 1$. It follows from \eqref{lang-4} that the number of occurrences of $1$ in $x_1\ldots x_n$ decreases exponentially fast with $n\to\infty$. Hence $x_i=0$ for all $i$ in a set of asymptotic density $1$. This means that $x$ is a generic point for $\mu_0$.
\end{proof}

Any finite number of ergodic measures can be achieved by a similar construction.

\section{Open questions}\label{sec:openqu}

We close this paper by offering two questions for further research:
\begin{enumerate}
  \item Assuming that $T$ is closeable with respect to a linkable set  $K\subset\Per(T)$ and $\CM(T)$ is not a single periodic orbit, does $T$ have positive topological entropy?
  \item Assuming that $\MT(X)$ is the Poulsen simplex, is it true that every $\mu\in\MT(X)$ has a generic point?
\end{enumerate}

%\appendix

\section*{Acknowledgements} We are grateful to Christian Bonatti, Keith Burns, Vaughn Climenhaga, Sylvain Crovisier, Lorenzo J. D\'{\i}az, Martha {\L}{\c{a}}cka, Piotr Oprocha, and Paulo Varandas   for numerous discussion and remarks related to this paper.

\end{document}